\documentclass[12pt]  {amsart} %{article}
\usepackage{amsmath,latexsym,amssymb,amsmath,
amscd,amsthm,amsxtra}

\numberwithin{equation}{section}

\newtheorem{example}{Example}[section]
\newtheorem{definition}{Definition}[section]
\newtheorem{theorem}{Theorem}[section]
\newtheorem{proposition}{Proposition}[section]
\newtheorem{lemma}{Lemma}[section]
\newtheorem{corollary}{Corollary}[section]

\newtheorem{remark}{\textbf{Remark}}[section]

\def\<{\langle}
\def\>{\rangle}
\def\div{{\rm div}}

\begin{document}

\title[ Anisotropic Minkowski Problem]{On an Anisotropic Minkowski Problem}
\thanks{C. X. is supported by China Scholarship Council.}
\author{Chao Xia}\address{Albert-Ludwigs-Universit\"at Freiburg, Mathematisches Institut,  Eckerstr. 1, 79104 Freiburg, Germany}
\email{chao.xia@math.uni-freiburg.de}

\date{March 6, 2012}

%\maketitle

\begin{abstract}{ In this paper, we study the anisotropic Minkowski problem. It is a problem of prescribing the anisotropic Gauss-Kronecker curvature for a closed strongly convex hypersurface in Euclidean space as a function on its anisotropic normals in relative or Minkowski geometry. We first formulate such problem to a Monge-Amp\'ere type equation on the anisotropic support function and  then prove the existence and uniqueness of the admissible solution to such equation. In conclusion, we give an affirmative answer to the anisotropic Minkowski problem.}

\



\end{abstract}

\medskip

\maketitle

\section{Introduction}
The Minkowski problem is a well known problem in the classical differential geometry: given a positive function $K$ on $\mathbb{S}^n$, can one find a closed strongly convex hypersurface whose Gauss-Kronecker curvature is given by $K$ as a function on its normals? This problem has been solved by the works of Minkowski \cite{Mi}, Alexandrov \cite{Al}, Lewy \cite{Le}, Nirenberg \cite{Ni}, Pogorelov \cite{Po} and eventually Cheng-Yau \cite{CY}. As is well known, the solvability of the Minkowski problem is  equivalent to that of a Monge-Amp\'ere equation. The analytic method of Nirenberg, Pogorelov and Cheng-Yau to the Minkowski problem led to significant development of the theory of the Monge-Amp\'ere equation. Many generalized problems around convex hypersurfaces with other prescribed curvature functions were considered intensively in recent years, see e.g. \cite{GG} and \cite{GM}. Most of them can be formulated as fully nonlinear elliptic equations. We refer to the lecture note of Guan \cite{Gu} for a complete description of the fully nonlinear elliptic equations arising from geometry, particularly the Minkowski problem.

This paper aims at an analogous investigation on a Minkowski type problem which incorporates some anisotropy, which we will call the anisotropic Minkowski problem. Such anisotropy was considered in two aspects of research area. One can be dated back to Wulff \cite{Wu}, who initiated the study of an anisotropic convex functional in the theory of crystals and materials. The minimizer of the anisotropic convex functional is called Wulff shape. After that, so many works appeared to study the crystalline variational problem and Wulff shape, see e.g. \cite{Ta} and reference therein.  The other was first studied by Minkowski \cite{Mi2} in  so-called relative or Minkowski differential geometry, where the role of sphere can be assumed by some other smooth convex hypersurfaces, in contrast with Euclidean geometry. The questions arising from  relative or Minkowski geometry were intensively investigated by a number of mathematicians, see e.g. \cite{BF}, \cite{Bu}, \cite{Re}, \cite{G}, \cite{Th}, \cite{An} and so on.

In relative or Minkowski geometry, we are always given a Minkowski norm.
\begin{definition}\label{Min}
A function $F:\mathbb{R}^{n+1}\to [0,+\infty)$ is called a Minkowski norm if 
\begin{itemize}
\item[(i)] $F$ is a norm of $\mathbb{R}^{n+1}$, i.e.,  $F$ is a convex, $1$-homogeneous function satisfying $F(x)>0$ when $x\neq 0$;
\item[(ii)] $F\in C^\infty(\mathbb{R}^{n+1}\setminus\{0\})$;
\item[(iii)] $F$ satisfies a uniformly elliptic condition: $\hbox{Hess} (\frac12 F^2)$ is positive definite in $\mathbb{R}^{n+1}\setminus \{0\}$. 
\end{itemize}\end{definition}
 There is an important smooth convex hypersurface corresponding to $F$, the Wulff shape $\mathcal{W}_F=\{x\in\mathbb{R}^{n+1}|F(x)=1\}$. We can view it as the unit sphere in the Minkowski space $(\mathbb{R}^n,F)$. It plays the role in relative geometry as the standard sphere in classical Euclidean geometry.
  
For an $n$ -dimensional oriented hypersurface $M$ in $\mathbb{R}^{n+1}$, the area in relative geometry should be computed as $\int_M d\mu_F=\int_M F^0(\bar{\nu}) d\mathcal{H}^n$, with $\bar{\nu}$ the standard normal and $F^0$ the dual norm of $F$. The anisotropic Gauss map (anisotropic normal) of $M$ is a map from $M$ to the Wulff shape  $\mathcal{W}_F$. Using such anisotropic Gauss map,  the anisotropic curvatures can be well defined. The major difference between relative geometry  and Euclidean geometry lies on the fact that the metric we consider in $\mathbb{R}^{n+1}$ is not Euclidean metric any more, but a new one $G$ instead, depending on the second derivative of $F$ (see \eqref{G} below), which varies  from point to point.  As well, the metric on $M$ is chosen as $g$, the restriction of $G$ on $M$, but not that of the Euclidean metric. This arises serious complications and difficulties for the geometric problems.  We will review the definitions and foundations of relative geometry  in Section 2. The setting here is largely motivated by the work of Andrews \cite{An}, though the notations appear differently in this paper. 

As in the classical differential geometry, when $M$ is a closed strongly convex hypersurface in $\mathbb{R}^{n+1}$, the anisotropic Gauss map defines a diffeomorphism between $M$ and $\mathcal{W}_F$. Therefore, $M$ can be reparametrized  by the inverse anisotropic Gauss map. In turn, the anisotropic curvatures can be viewed as functions on $\mathcal{W}_F$. 
In particular, the anisotropic Gauss-Kronecker curvature $K(z)$ for $z\in \mathcal{W}_F$ must satisfy (see \eqref{necc} below) \begin{eqnarray}\label{necce}
\int_{\mathcal{W}_F} G(z)(z, E^\alpha)\frac{1}{K(z)} d\mu_F=0,\forall \alpha=1,\cdots, n+1,
\end{eqnarray}
where  $E^\alpha$ is the standard coordinate vectors in $\mathbb{R}^{n+1}$ .

The anisotropic Minkowski problem is the converse of the previous statement, namely, given a positive function $K$ on $\mathcal{W}_F$, can one find a closed strongly convex hypersurface whose anisotropic Gauss-Kronecker curvature is given by $K$ as a function on its anisotropic normals?

In this paper, we solve the anisotropic Minkowski problem. The main result is the following
\begin{theorem}\label{thm} Let $F$ be a Minkowski norm in $\mathbb{R}^{n+1}$. Let $K$ be a positive function in $C^k(\mathcal{W})$  with $k\geq 2$ and satisfy the condition \eqref{necce}. Then there is a $C^{k+1,\alpha} (\forall 0<\alpha<1)$ closed strongly convex hypersurface $M$ in $\mathbb{R}^{n+1}$ whose anisotropic Gauss-Kronecker curvature is $K$ as a function on its anisotropic normals. Moreover, $M$ is unique up to translations.
 \end{theorem} 
 \begin{remark}
 It can be seen from the proof that the smoothness of $F$,(ii) in Definition \ref{Min},  can be assumed only in $C^{k+3}(\mathbb{R}^{n+1}\setminus\{0\})$.
 \end{remark}
As in the classical Minkowski problem, we can reduce Theorem \ref{thm} to the solvability of a Monge-Amp\'ere type equation on the anisotropic support function $S$,
\begin{eqnarray}\label{0eq}
\det (S_{ij}-\frac12Q_{ijk}S_k+S\delta_{ij})=\frac{1}{K}\quad \hbox{ on }\mathcal{W}.
\end{eqnarray}

Here we give two remarks for \eqref{0eq}. First, the covariant derivatives of $S$ are all corresponding to the Riemannian metric $g$, which is the restriction of $G$ on $\mathcal{W}_F$, but not restriction  of the Euclidean metric on $\mathcal{W}_F$. Second, $Q_{ijk}$ is a $3$-tensor on $\mathcal{W}_F$, which corresponds to the third derivative of $F$. Hence in general, it does not vanish.  In fact, it vanishes if and only if $F$ is quadratic, in which case $\mathcal{W}_F$ is ellipsoid. This causes major difficulty when we prove a priori estimates for the Monge-Amp\'ere equation. For the detailed derivation of \eqref{0eq} and the reduction of the solvability of the anisotropic Minkowski problem to that of \eqref{0eq}, see Section 3.

As usual, we will apply the method of continuity to  solve \eqref{0eq}. 
The first issue is the a priori estimates for solutions of \eqref{0eq}. By modifying Cheng-Yau's proof in \cite{CY}, we are able to give a uniformly upper bound of the anisotropic outer radius of $M$, which leads to the $C^0$ estimate. To proceed to higher order estimates, it seems necessary to derive a uniformly positive lower bound of the anisotropic inner radius of $M$. Cheng-Yau's proof is highly nontrivial and seems not applicable. We  apply instead a new idea, which combines an inequality of Andrews \cite{An} and a uniformly positive lower  bound of  the anisotropic outer radius, to give an explicit uniformly  positive lower bound of the anisotropic inner radius of $M$.  The difficulty arises when we deal with the $C^2$ estimate. In the classical one, there is no gradient term in the equation. Also the simple representation of Gauss equation on the sphere makes the $C^2$ estimate possible without deriving $C^1$ estimate. Our situation is much more complicated due to both the gradient term in \eqref{0eq} and more complicated Gauss equation (see Lemma \ref{lem2-1}). It seems indispensable to derive the $C^1$ estimate first.
Fortunately, since we already have the positive lower and upper bound of $S$, we can choose an auxiliary function as the sum of gradient part and some lower order part, explicitly, we choose $W=\log |\nabla S|^2+e^{\alpha (m_2-S)}$, where $m_2$ is the upper bound of $S$, $\alpha$ is some  large constant. With this choice, we are able to use the maximum principle to obtain bounds for $W$ and then bounds for $|\nabla S|$.
The $C^2$ estimate cannot be proved as usual  either. Here we adopt some idea of Yau's  proof in \cite{Ya} for Calabi conjecture and Guan-Li's  proof \cite{GL} for more general complex Monge-Amp\'ere equation. We choose an auxiliary function $\Phi=\log (a+\Delta S)+e^{\beta (m_2-S)}$, where $a,\beta$ are some constant. Then it is possible to derive bound for $\Phi$ and then bound for $|\nabla^2 S|$.

Besides the a priori estimates  for solutions of \eqref{0eq}, we also need to prove the openness of sets of solutions. Thus it is necessary to study the linearized operator $L_S$ of $S\mapsto \det (S_{ij}-\frac12Q_{ijk}S_k+S\delta_{ij})$. In the classical proof, the divergence free property of Newton transformation $\frac{\partial \det (u_{ij})}{\partial u_{ij}}$ is quite important to prove the self-adjointness of $L_S$. Here such property fails. However, by using the explicit Gauss equation, we still be able to prove the self-adjointness of $L_S$ with respect to the anisotropic measure $d\mu_F$ (see Lemma \ref{adjoint}). The kernel of $L_S$ is explicitly derived as well (see Lemma \ref{ker}). With these at hand, the openness can be proved in a standard way.
The uniqueness part in  Theorem \ref{thm} follows easily from Lemma \ref{adjoint} and \ref{ker}.

\

The organization of our paper is as follows. In Section 2, we first review some basic concepts and foundations in relative geometry, such as the anisotropic Gauss map, the anisotropic second fundamental form, the anisotropic curvatures and so on. We then  derive the anisotropic version of Gauss-Weingarten formulas and Gauss-Codazzi equations. The relationship between the anisotropic volume form $d\mu_F$ and the induced volume form of the metric $g$ is also revealed. In Section 3, we formulate the anisotropic Minkowski problem as the Monge-Amp\'ere equation \eqref{0eq} on  the anisotropic support function in details.  After such preparations, we shall establish the a priori estimates for the equation \eqref{0eq} in Section 4. In the last section, we derive the openness of the set of solutions and then give a complete proof of Theorem \ref{thm}.

\

\noindent\textbf{Notation:} Throughout this paper, the Latin alphabet $i,j,k,\cdots$ denotes indices from $1$ to $n$ and the Greek alphabet $\alpha,\beta,\gamma,\cdots$ denotes indices from
$1$ to $n+1$.  We will always use the Einstein summation convention. 

 \
 
 \section{Relative geometry of hypersurface}

Let $F$ be a Minkowski norm in $\mathbb{R}^{n+1}$. The dual norm of $F$ is defined as  $$F^0(\xi):=\sup_{x\neq 0}\frac{\langle x,\xi\rangle}{F(x)},\quad \xi\in \mathbb{R}^{n+1}.$$ By the assumption on $F$, we see that $F^0$ is a also norm that at least belongs to $C^2(\mathbb{R}^{n+1}\setminus\{0\})$.
The following properties is quite simple consequences of $1$-homogeneous of $F$ and $F^0$. We omit the proof here.
\begin{proposition}\label{property1}
\
\begin{itemize}
\item[(i)] $ \frac{\partial F}{\partial x^\alpha}(x)x^\alpha=F(x),\quad\frac{\partial F^0}{\partial \xi^\alpha}(\xi)\xi^\alpha=F^0(\xi);$
\item[(ii)] $\frac{\partial^2 F}{\partial x^\alpha\partial x^\beta}(x)x^\alpha=0,\quad \frac{\partial^2 F^0}{\partial \xi^\alpha\partial \xi^\beta}(\xi)\xi^\alpha=0,\hbox{ for }x,\xi\neq 0, \quad\forall \beta=1,\cdots,n+1;$
\item[(iii)] $F(D F^0(\xi))=1,\quad F^0(D F(x))=1,\hbox{ for }x,\xi\neq 0;$
\item[(iv)]  $F^0(\xi)D F(DF^0(\xi))=\xi,\quad F(x)D F^0(D F(x))=x,\hbox{ for }x,\xi\neq 0.$
\end{itemize}
Here $D F=(\frac{\partial F}{\partial x^1},\cdots,\frac{\partial F}{\partial x^{n+1}})$ and $D F^0=(\frac{\partial F^0}{\partial \xi^1},\cdots,\frac{\partial F^0}{\partial \xi^{n+1}})$.
\end{proposition}

A smooth convex hypersurface in $\mathbb{R}^{n+1}$ corresponding to $F$ is the Wulff shape $$\mathcal{W}_F:=\{x\in \mathbb{R}^{n+1}| F(x)=1\}.$$   Conversely, a smooth convex hypersurface $\mathcal{M}$ in $\mathbb{R}^{n+1}$ determines uniquely a convex function $F$ such that the Wulff shape corresponding to such $F$ is $\mathcal{M}$. 
Wulff shape plays the fundamental role as a comparison body in the relative differential geometry.

For an oriented $n$-dimensional hypersurface $M$ in $\mathbb{R}^{n+1}$,
 The function $F$ defines an anisotropic area functional
 \begin{equation*}
|M|_F:=\int_M F^0(\bar{\nu}) d\mathcal{H}^n,
\end{equation*}
where $\bar{\nu}$ and $\mathcal{H}^n$ denote the standard unit outer normal to $M$ and the $n$-dimensional Hausdorff measure respectively.  We denote by $d\mu_F=F^0(\bar{\nu}) d\mathcal{H}^n$ and call it the anisotropic measure. 

The unit anisotropic  outer normal is defined by $$\nu_F:=DF^0(\bar{\nu}).$$ It is easy to see from Proposition \ref{property1} (iii) that $\nu_F\in \mathcal{W}_F$. We call $\nu_F: M\to \mathcal{W}_F$  the anisotropic Gauss map. 

From now on we will omit the subscript $F$ for simplicity, i.e., $d\mu=d\mu_F$, $\nu=\nu_F$ and $\mathcal{W}=\mathcal{W}_F$, etc..

Since $\hbox{Hess}(\frac12 F^2)$ is positive definite, we can consider $\mathbb{R}^{n+1}$ as a Riemannian manifold equipped with a metric 
\begin{eqnarray}\label{G}
&& G(x)(\xi,\eta):=\sum_{\alpha,\beta=1}^{n+1} G_{\alpha\beta}(x)\xi^\alpha \eta^\beta=\sum_{\alpha,\beta=1}^{n+1}\frac{\partial^2(\frac12F^2)}{\partial x^\alpha \partial x^\beta}(x)\xi^\alpha \eta^\beta
\end{eqnarray}
 for $x\in\mathbb{R}^{n+1}$, $\xi,\eta\in T_x \mathbb{R}^{n+1}$.  It is easy to see from Proposition \ref{property1} (i) and (ii) that 
\begin{eqnarray}&&\label{G1}G(x)(x,x)=1, \quad G(x)(x,\xi)=0\hbox{ for any }x\in \mathcal{W}\subset\mathbb{R}^{n+1},\quad \xi\in T_x \mathcal{W}. 
\end{eqnarray} Similarly, for an oriented $n$-dimensional hypersurface $M$ in $\mathbb{R}^{n+1}$, we have $$G(\nu)(\nu,\nu)=1,\quad G(\nu)(\nu,\xi)=0\hbox{ for }x\in M,\quad \nu=\nu(x),\quad \xi\in T_x M.$$ 

We define $$g(x):=G(\nu(x))|_{T_x M}, \quad x\in M$$ as a Riemannian metric on $M\subset \mathbb{R}^{n+1}$.  This Riemannian metric will play a fundamental role in relative geometry of hypersurface. Note that in classical Euclidean geometry the metric on a hypersurface is the restriction of the Euclidean metric. This is the major difference between the relative geometry and Euclidean geometry.

Since $F$ is not quadratic, the third derivative of $F$ does not vanish. We denote 
\begin{eqnarray*}
Q(x)(\xi,\eta,\zeta):=\sum_{\alpha,\beta,\gamma=1}^{n+1} Q_{\alpha\beta\gamma}(x)\xi^\alpha \eta^\beta \zeta^\gamma=\sum_{\alpha,\beta,\gamma=1}^{n+1} \frac{\partial^3(\frac12F^2)}{\partial x^\alpha \partial x^\beta \partial x^\gamma}(x)\xi^\alpha \eta^\beta \zeta^\gamma,
\end{eqnarray*}
for $x,\xi,\eta,\zeta\in\mathbb{R}^{n+1}$.
It follows again from Proposition \ref{property1} (i) and (ii) that 
\begin{eqnarray}\label{Q1}
Q(x)(x,\xi,\eta)=0.
\end{eqnarray}

\

To study the relative geometry of hypersurface, it is indispensable to define the anisotropic second fundamental form. 

For $X\in M\subset \mathbb{R}^{n+1}$, choose local coordinate  $\{y^\alpha\}_{\alpha=1}^{n+1}$ in $\mathbb{R}^{n+1}$, such that $\{\frac{\partial}{\partial y^\alpha}\}_{\alpha=1}^{n+1} $ are tangent to $M$ and $\frac{\partial}{\partial y^{n+1}}=\nu$ is the unit anisotropic  outer normal of $M$. By identification $\partial_i=\frac{\partial}{\partial y^i}=\partial_i X$ for $i=1,\cdots,n$. Let $g_{ij}(X)=g(\nu)(\partial_i X,\partial_j X)$ be the Riemannian metric on $M$. Denote by $g^{ij}$ the inverse of $g_{ij}$.

Given the standard volume form $\Omega$ (Lebesgue measure)  in $\mathbb{R}^{n+1}$,  the anisotropic measure on $M$ can be interpreted as \begin{eqnarray}\label{vol}
d\mu=\Omega(\nu, \partial_1,\cdots,\partial_n) dy^1\cdots dy^n. \end{eqnarray}
 This follows from the fact that $d\mu=F^0(\bar{\nu}) \Omega(\bar{\nu}, \partial_1,\cdots,\partial_n) dy^1\cdots dy^n $ and $\nu=F^0(\bar{\nu})\bar{\nu}+\hbox{tangent part}$, due to Proposition \ref{property1} (i).

 The anisotropic second fundamental form is defined by $$h_{ij}(X)=-G(\nu)(\nu,\partial_i\partial_j X).$$ It is a symmetric $2$-tensor on $M$.

With the Riemannian metric $g$ and the anisotropic second fundamental form, we can define  the anisotropic principle curvatures and the Gauss-Kronecker curvature.
\begin{definition} The eigenvalues of the anisotropic second fundamental form $h_{ij}$ with respect to the metric $g_{ij}$ (i.e. the eigenvalues of the matrix $g^{jk}h_{ij}$) are called the anisotropic principle curvatures. The inverse of the anisotropic principle curvatures are called anisotropic principle radii.
\end{definition}

\begin{definition}
The anisotropic mean curvature is $H=g^{ij}h_{ij}$. The anisotropic Gauss-Kronecker curvature is $K=\det(g^{jk}h_{ij})$.
\end{definition}

\begin{proposition}\label{convex}
A hypersurface $M$ is convex (strongly convex resp.) if and only if $(h_{ij})\geq 0$ ($>0$  resp.).
\end{proposition}
\begin{proof}we just need to observe that \begin{eqnarray*}
h_{ij}=\frac{1}{F^0(\bar{\nu})}\bar{h}_{ij},
\end{eqnarray*}
where $\bar{h}_{ij}$ is the standard second fundamental form $\bar{h}_{ij}=\langle\bar{\nu},\partial_i\partial_j X\rangle_{\mathbb{R}^{n+1}}$.

Indeed, 
 in view of $\nu=DF^0(\bar{\nu})$, we have
\begin{eqnarray*}
h_{ij}&=& -G(\nu)(\nu,\partial_i\partial_j X)=-DF(DF^0(\bar{\nu}))\cdot\partial_i\partial_j X\\&=&-\frac{\bar{\nu}}{F^0(\bar{\nu})}\cdot\partial_i\partial_j X=\frac{1}{F^0(\bar{\nu})}\bar{h}_{ij}.\end{eqnarray*}
Here we have used Proposition \ref{property1} (iii) and (iv).
\end{proof}

Similarly as in the classical theory of hypersurface, we have the following Gauss and Weingarten formulas, Gauss and Codazzi equations.
\begin{lemma} \label{lem2-1}
 \begin{eqnarray}\label{Gauss}
\partial_i\partial_j X=-h_{ij}\nu+\nabla_{\partial_i } \partial_j +A_{ij}^k  \partial_k X; \hbox{ (Gauss formula)}\end{eqnarray}
\begin{eqnarray}\label{Weigarten}
\partial_i \nu=g^{jk}h_{ij}\partial_k X; \hbox{ (Weingarten formula) }
\end{eqnarray}
 \begin{eqnarray}\label{Gausseq}
R_{ijkl}&=&h_{ik}h_{jl}-h_{il}h_{jk}+\nabla_{\partial_l } A_{jki}-\nabla_{\partial_k } A_{jli}\\&&+A_{jk}^m A_{mli}-A_{jl}^m A_{mki}; \hbox{ (Gauss equation) }\nonumber
\end{eqnarray}
\begin{eqnarray}\label{Codazzi}
h_{ijk}+h_j^l A_{lki}=h_{ikj}+h_k^l A_{lji}. \hbox{ (Weingarten formula) }
\end{eqnarray}

Here $\nabla$ is the Levi-Civita connection of the Riemannian metric $g$, $R$ is the Riemannian curvature tensor of $g$, $A$ is a $3$-tensor
\begin{eqnarray}\label{A}
A_{ijk}=-\frac12\left(h_i^l Q_{jkl}+h_j^l Q_{ilk}-h_k^l Q_{ijl}\right),
\end{eqnarray} where $Q_{ijk}=Q(\nu)(\partial_i X, \partial_j X, \partial_k X)$.
\end{lemma}

\begin{proof}
Taking derivative of the equation $G(\nu)(\nu,\nu)=1$, $G(\nu)(\nu,\partial_j X)=0$ and using \eqref{Q1}
we have
\begin{eqnarray*}
G(\nu)(\partial_i \nu,\nu)=0, 
\end{eqnarray*}
\begin{eqnarray*}
G(\nu)(\partial_i \nu,\partial_i X)+G(\nu)(\nu,\partial_i \partial_j X)=0,
\end{eqnarray*}
which implies the Weingarten formula \eqref{Weigarten}.

To verify the Gauss formula \eqref{Gauss}, it is sufficient to give the explicit formula \eqref{A} for $A$. Denote $\Gamma_{ij}^k$ the Christoffel symbol with respect to  $\nabla$.
Taking derivative of the equation $g_{ij}=G(\nu)(\partial_i X,\partial_j X)$, we have
\begin{eqnarray}\label{weig}
&&\partial_k g_{ij}= G(\nu)(\partial_k \partial_i X, \partial_j X)+G(\nu)( \partial_i X, \partial_k\partial_j X)+Q(\nu)(\partial_k\nu, \partial_i X, \partial_j X)
\\&=& (\Gamma_{ik}^l+A_{ik}^l)g_{jl}+(\Gamma_{jk}^l+A_{jk}^l)g_{il}+h_{k}^l Q_{ijl}.\nonumber
\end{eqnarray}
Note that $\Gamma_{ij}^k=\frac12 g^{kl}(\partial_i g_{jl}+\partial_j g_{il}-\partial_l g_{ij})$. Then
\eqref{A} follows easily from \eqref{weig}.

Taking covariant derivative of the Weingarten formula \eqref{Weigarten} , we have
\begin{eqnarray*}
\nabla_{\partial_j} \nabla_{\partial_i} \nu=( h^k_{i,j}+h_i^l A_{jl}^k)\partial_k X+ \hbox{ anisotropic normal part }.
\end{eqnarray*}
Then the Codazzi equation \eqref{Codazzi} follows from the symmetry $\nabla_{\partial_j} \nabla_{\partial_i} \nu=\nabla_{\partial_i} \nabla_{\partial_j} \nu$.

We are remained with the verification of the Gauss equation \eqref{Gausseq}. We choose the normal coordinate at some point $p_0$ with respect to $g$, such that $g_{ij}(p_0)=\delta_{ij}$ and $\Gamma_{ij}^k(p_0)=0$.
Taking derivative of the Gauss formula \eqref{Gauss}, we have at $p_0$,
\begin{eqnarray*}
&&\partial_l \partial_j \partial_i X=\nabla_{\partial_l}\left[-h_{ij}\nu+(\Gamma_{ij}^k+A_{ij}^k)\partial_k X\right]\\
&=& \left[-h_{ij}h_l^{k}+\partial_l(\Gamma_{ij}^k+A_{ij}^k)+A_{ij}^mA_{lm}^k\right]\partial_k X+\hbox{  anisotropic normal part  }.
\end{eqnarray*}
Hence \begin{eqnarray}\label{gau}
&&0=\partial_l \partial_j \partial_i X-\partial_j \partial_l \partial_i X\\&=&\left[h_{il}h_j^k-h_{ij}h_l^{k}+\partial_l\Gamma_{ij}^k-\partial_j\Gamma_{il}^k+\partial_l A_{ij}^k-\partial_jA_{il}^k+A_{ij}^mA_{lm}^k-A_{il}^mA_{jm}^k\right]\partial_k X\nonumber\\
&&+\hbox{  anisotropic normal part }.\nonumber
\end{eqnarray}

By definition of Riemannian curvature, at $p_0$,
\begin{eqnarray*}
R_{ijkl}&=& g(\nu)(\nabla_{\partial_k} \nabla_{\partial_l} \partial_ j X-\nabla_{\partial_l} \nabla_{\partial_k} \partial_j X, \partial_i X)\\
&=&\partial_k\Gamma_{lj}^i-\partial_l\Gamma_{kj}^i.	
\end{eqnarray*}
Now it follows from \eqref{gau} that
\begin{eqnarray*}
R_{ijkl}=h_{ik}h_{jl}-h_{il}h_{jk}+\partial_l  A_{jki}-\partial_k  A_{jli}+A_{jk}^m A_{mli}-A_{jl}^m A_{mki}.
\end{eqnarray*}
Since both sides are tensors, \eqref{Gausseq} holds at every point. We complete the proof.
\end{proof}

\begin{remark} We see from \eqref{A} that the $3$-tensor $A$ is symmetric in the first two indices, namely, $A_{ijk}=A_{jik}$. However, it is not totally symmetric in general. Indeed, it can be shown that $A_{ijk}+A_{ikj}=-h_i^l Q_{ljk}$.\end{remark}

We will compute in the following example  the geometry quantities defined above for the special case $M=\mathcal{W}$. It shows that $\mathcal{W}$ plays the same role in relative geometry as the standard sphere in classical Euclidean geometry. 
\begin{example}\label{exam} Consider the hypersurface $M=\mathcal{W}\subset\mathbb{R}^{n+1}$, the Wulff shape. In this case, the position vector and the unit anisotropic outer normal coincide, i.e., $X=\nu(X)$ for $X\in\mathcal{W}$. Hence $\partial_i X=\partial_i \nu=h_i^j \partial_j X$, which implies that $h_i^j=\delta_{ij}$ and  in turn, $h_{ij}=g_{ij}$. In this case $A_{ijk}=-\frac12 Q_{ijk}$ and it is totally symmetric for all the indices. It is easy to see that \begin{eqnarray}\label{4Q}
\nabla_{\partial_l}Q_{ijk}=\nabla_{\partial_k}Q_{ijl}.
\end{eqnarray} Hence the Gauss equation \eqref{Gausseq} can be easier written as
\begin{eqnarray}\label{Gausseq1}
R_{ijkl}&=& g_{ik}g_{jl}-g_{il}g_{jk}+\frac14 Q_{jk}^m Q_{mli}-\frac14 Q_{jl}^m Q_{mki}.
\end{eqnarray}
\end{example}

\

The following lemma states that the difference between anisotropic volume form and induced volume form by $g$ has a natural relationship with the tensor $A$. It is a simple but quite important observation in the relative geometry.
\begin{lemma}\label{volume}
Let $dV_g$ be the induced volume form of $M$ equipped with $g$. Assume that $d\mu=\varphi dV_g$. Then
\begin{equation*}
\partial_i \log\varphi=A_{ij}^j= g^{jk}A_{ijk}.
\end{equation*}
\end{lemma}
\begin{proof}
In local coordinates, $$\varphi=\frac{\Omega(\nu,\partial_1 X, \cdots, \partial_n X)}{\sqrt{\det (g)}}.$$
We compute that
\begin{eqnarray*}
&&\partial_i \log\Omega(\nu,\partial_1 X, \cdots, \partial_n X)\\&=&\frac{1}{\Omega(\nu,\partial_1 X, \cdots, \partial_n X)}\sum_{j=1}^n \Omega(\nu, \partial_1 X, \cdots,\partial_i\partial_jX,\cdots,\partial_n X)\\&=&(\Gamma_{ij}^j+A_{ij}^j).
\end{eqnarray*}
On the other hand,  since $g$ is a Riemannian metric on $M$, 
\begin{eqnarray*}
\partial_i \log \sqrt{\det (g)}=\frac12 g^{jk}\partial_i g_{jk}=\Gamma_{ij}^j.\\
\end{eqnarray*}
Therefore, we have \begin{eqnarray*}
\partial_i \log \varphi=A_{ij}^j.
\end{eqnarray*}
\end{proof}

\

\section{Formulation of the anisotropic Minkowski problem}
Let $M$ be an $n$-dimensional closed, strongly convex hypersurface in $\mathbb{R}^{n+1}$. Since the map $\mathbb{S}^n\to\mathcal{W}: \bar{\nu}\mapsto \nu=DF^0(\bar{\nu})$ defines a  nondegenerate diffeomorphism between $\mathbb{S}^n$ and $\mathcal{W}$, we easily see that the anisotropic Gauss map $\nu:M\to \mathcal{W}$ is everywhere nondegenerate diffeomorphism. We can use it reparametrize the convex hypersurface, i.e.

$$X:\mathcal{W}\to M\subset\mathbb{R}^{n+1}, \quad X(z)=X(\nu^{-1}(z)), \quad z\in\mathcal{W}.$$ By virtue of Proposition \ref{convex}, the anisotropic Gauss-Kronecker curvature of $M$ is positive.
With this parametrization, it can be viewed as a positive function $K(\nu^{-1}(z))$ on the Wulff shape $\mathcal{W}$.

The anisotropic Minkowski problem is the anisotropic version of Minkowski problem in classical geometry. Namely, it is a problem of prescribing the anisotropic Gauss-Kronecker curvature on the anisotropic normals of a closed strongly convex hypersurface. We state this problem as follows:

\

\noindent\textbf{Anisotropic Minkowski problem:} Given a positive function $K$ on $\mathcal{W}$, is there a closed strongly convex hypersurface whose anisotropic Gauss-Kronecker curvature is $K$ as a function on its anisotropic normals?

\

A necessary condition to this problem is that $K$ must satisfy
\begin{eqnarray}\label{necc}
\int_{\mathcal{W}} G(z)(z,E^\alpha)\frac{1}{K(z)} d\mu=0,\quad \forall \alpha=1,\cdots,n+1,
\end{eqnarray}
where $E^\alpha$ denote the standard $\alpha$-th coordinate vector in $\mathbb{R}^{n+1}$.
In fact, in view of \eqref{vol} and by using divergence theorem, we have
\begin{eqnarray*}
&&\int_{\mathcal{W}} G(z)(z,E^\alpha)\frac{1}{K(z)} d\mu\\&=&\int_{M} G(\nu)(\nu,E^\alpha)d\mu(M)=
\int_M G(\nu)(\nu,E^\alpha)\Omega(\nu,\partial_1,\cdots,\partial_n)dy^1\cdots dy^n\\&=&\int_M \Omega(E^\alpha,\partial_1,\cdots,\partial_n)dy^1\cdots dy^n=\int_{M} \langle \bar{\nu}, E^\alpha \rangle_{\mathbb{R}^{n+1}} dvol(M)\\
&=&\int_{\bar{M}} \div(E^\alpha) d\mathcal{H}^{n+1}=0.
\end{eqnarray*}
Here $\bar{M}$ is the body enclosed by $M$, and $dvol(M)=\Omega(\bar{\nu},\partial_1,\cdots,\partial_n)dy^1\cdots dy^n$ is the induced volume form of the Euclidean metric in $\mathbb{R}^{n+1}$.

As in the classical Minkowski problem, we will reduce the solvability of the anisotropic Minkowski problem to that of a fully nonlinear elliptic equation of a suitable support function. First of all, let us introduce the anisotropic support function.

The anisotropic support function of $M$ is defined as \begin{eqnarray*}
S(z)=\sup_{y\in M} G(z)(z,y)=G(z)(z,X(z)), \hbox{ for }z\in \mathcal{W}.
\end{eqnarray*}

We will compute the metric $g$ and the anisotropic second fundamental form $h$ of $M$ in terms of the anisotropic support function $S$.
 
Let $z\in\mathcal{W}$. Choose an orthonormal basis $\{e_i\}_{i=1}^n$ of $T_z \mathcal{W}=T_{X(z)} M$ with respect to the Riemannian metric $g$. Denote by $\nabla$ the covariant derivative with respect to $g$ on $\mathcal{W}$.
Taking the first covariant derivative of $S$, we have
\begin{eqnarray*}
\nabla_{e_i} S(z)&=& G(z)(\nabla_{e_i} z, X(z))+G(z)(z, \nabla_{e_i}  X(z))+Q(z)(\nabla_{e_i} z, z, X(z))\\&=&G(z)(\nabla_{e_i} z, X(z)),
\end{eqnarray*}
The last two terms vanish due to \eqref{G1} and \eqref{Q1}.

Taking the second covariant derivative of $S$, by using Gauss formula \eqref{Gauss} we have (we compute at normal coordinate of $g$, namely, $\nabla_{e_i} e_j=0$)
 \begin{eqnarray*}
\nabla_{e_i}\nabla_{e_j} S(z)&=&e_ie_j G(z)(z,X(z))=e_i(G(z)(e_j z,X))\\
&=&G(z)(e_i e_j z,X(z))+G(z)(e_j z,e_i X)+ Q(e_i z,e_j z,X(z))\\
&=&-\delta_{ij}G(z)(z,X(z))-\frac12Q(e_i z,e_j z,e_k z)G(z)(e_k, X(z))\\
&&-G(z)(z,(-h_{ij}(X)z))+Q(e_i z,e_j z,e_k z)G(z)(e_k, X(z))\\
&=&-\delta_{ij} S(z)+h_{ij}(X(z))+\frac12Q_{ijk}\nabla_{e_k}S(z).
\end{eqnarray*}
Here we also used the observation in Example \ref{exam}.

For simplicity, we use the abbreviation $S_i, S_{ij}$ to denote the covariant derivative of $g$.
Thus it follows from previous computation that the anisotropic second fundamental form of $M$ has the formula
\begin{eqnarray*}
h_{ij}(X(z))=S_{ij}(z)-\frac12Q_{ijk}S_k(z)+\delta_{ij} S(z),\quad\forall z\in\mathcal{W}.
\end{eqnarray*}

To compute the metric $g$ of $M$, we use the Weingarten formula \eqref{Weigarten},
$$e_i z=g^{jk}(X)h_{ij}(X)\nabla_{e_k} X,$$
from which we obtain
\begin{eqnarray*}
\delta_{ij}=g(z)(e_i z,e_j z)=h_{ik}g^{kl}h_{jl}.
\end{eqnarray*}
In turn, we have
\begin{eqnarray*}
g_{ij}=h_{ik}h_{jk}.
\end{eqnarray*}

Therefore, the anisotropic principal radii of $M$ are the eigenvalues of
$$g_{ik}h^{jk}=h_{ij}=S_{ij}(z)-\frac12Q_{ijk}S_k(z)+\delta_{ij} S(z),\quad\forall z\in\mathcal{W}.$$
The Gauss-Kronecker curvature is
$$K(z)=\frac{1}{\det(S_{ij}(z)-\frac12Q_{ijk}S_k(z)+\delta_{ij} S(z))},\quad\forall z\in\mathcal{W}.$$

\

In summary, we have proved the following proposition.
\begin{proposition}\label{P1}
Parametrizing a $C^2$ strongly convex hypersurface $M$ by the inverse anisotropic Gauss map over $\mathcal{W}$, we have that  the eigenvalue of $S_{ij}-\frac12Q_{ijk}S_k+S\delta_{ij} $ is the anisotropic principle radii of $M$. In particular, the  anisotropic Gauss-Kronecker curvature of $M$ satisfies 
\begin{eqnarray}\label{eq}
\det (S_{ij}-\frac12Q_{ijk}S_k+S\delta_{ij})=\frac{1}{K}\quad \hbox{ on }\mathcal{W}.
\end{eqnarray}
\end{proposition}

Conversely, 
given $S$ a $C^2$ function on $\mathcal{W}$ with $(S_{ij}-\frac12Q_{ijk}S_k+\delta_{ij} S)>0$, we are able to find a strongly convex hypersurface such that its anisotropic support function is $S$. 
\begin{proposition}\label{P2}
Any function $S\in C^2(\mathcal{W})$ with $(S_{ij}-\frac12Q_{ijk}S_k+S\delta_{ij}) >0$ is an anisotropic support function of  a $C^2$ strongly convex hypersurface $M$ in $\mathbb{R}^{n+1}$.
\end{proposition}
\begin{proof}
We extend $S$ to be a homogeneous function of degree one in $\mathbb{R}^{n+1}\setminus\{0\}$ by setting $S(x)=F(x)S\left(\frac{x}{F(x)}\right).$ Denote by $\nabla_{(\mathbb{R}^{n+1},G)}$ be the covariant derivative of $\mathbb{R}^{n+1}$ equipped with the metric $G$. Define 
\begin{eqnarray*}
M=\big\{\nabla_{(\mathbb{R}^{n+1},G)} S(x)|x\in \mathbb{R}^{n+1}\setminus\{0\}\big\}.
\end{eqnarray*}
Let $e_{n+1}=z$ be the position vector of $\mathcal{W}$ and $\{e_1,\cdots,e_n\}$ is a local orthonormal frame field with respect to $g$ on $\mathcal{W}$ such that  $\{e_1,\cdots,e_{n+1}\}$ is a positive oriented  orthonormal frame field with respect to $G$ in $\mathbb{R}^{n+1}$. 
Then it follows from the homogeneity of $S$ that for $y\in M$, there exists $z\in \mathcal{W}$, such that
\begin{eqnarray}\label{y}
y=y(z)=\nabla_{(\mathbb{R}^{n+1},G)} S(z)=\nabla_{e_i} S(z)e_i(z)+S(z) e_{n+1}(z) .
\end{eqnarray}
It is clear that
\begin{eqnarray}\label{cc}
e_i(e_{n+1})=e_i,\quad e_i(e_j)=\nabla_{e_i}e_j-\frac12 Q_{ijk}e_k-\delta_{ij}e_{n+1}.
\end{eqnarray}
Using \eqref{cc}, we compute the derivative of $y$ on $\mathcal{W}$ ( at normal coordinates, namely, $\nabla_{e_i} e_j=0$),
\begin{eqnarray}\label{xx}
e_j(y)&=&e_je_i( S) e_i+v_i e_j(e_i)+e_j( S)e_{n+1}+Se_j(e_{n+1})\\
&=&S_{ij}e_i+S_i(-\frac12 Q_{ijk}e_k-\delta_{ij}e_{n+1})+S_j e_{n+1}+Se_j\nonumber\\
&=&(S_{ij}-\frac12 Q_{ijk}S_k+\delta_{ij} S)e_i.\nonumber
\end{eqnarray}
Since $(S_{ij}-\frac12 Q_{ijk}S_k+\delta_{ij} S)>0$ by assumption, \eqref{xx} implies that the tangent space of $M$ at $y(z)$ is $\hbox{span}\{e_1(z),\cdots,e_n(z)\}$. Hence $e_{n+1}(z)=z$ is the anisotropic normal at $y(x)$. Now $\{e_1,\cdots,e_n,z\}$ gives an orientation of $M$. Also the map $y(z)=\nabla_{e_i} S(z)e_i(z)+S(z) e_{n+1}(z)$ is globally invertible and $M$ is an embedded hypersurface in $\mathbb{R}^{n+1}$.

In view of \eqref{y}, $S(z)=G(z)(z,y(z))$. It follows from the previous computation that the anisotropic  principle curvatures at $y(z)$ are the reciprocals of the eigenvalues of $(S_{ij}-\frac12 Q_{ijk}S_k+\delta_{ij} S)>0$. Therefore, $M$ is strongly convex and $S$ is its anisotropic support function.
\end{proof}

By virtue of Proposition \ref{P1} and \ref{P2},  the solvability of the  anisotropic Minkowski problem is reduced to that of the equation \eqref{eq}
on the Wulff shape $\mathcal{W}$ under the condition that $K\in C^k(\mathcal{W}),k\geq2$, $K>0$ and satisfies the equation \eqref{necc}.  Therefore, to prove  Theorem \ref{thm},  it is equivalent to prove the solvability of the equation \eqref{eq}.

\begin{definition} We call a solution $S$ of \eqref{eq} is an admissible solution if  the $n\times n$ matrix $S_{ij}-\frac12Q_{ijk}S_k+S\delta_{ij}$ is positive definite.
\end{definition}

\begin{theorem}\label{thm1}
Let $F$ be a Minkowski norm in $\mathbb{R}^{n+1}$.   Let $K$ be a positive function in $C^k(\mathcal{W})$  with $k\geq 2$ and satisfy the condition \eqref{necc}.  Then we can find an admissible solution $S\in C^{k+1,\alpha}(\mathcal{W})(\forall 0<\alpha<1)$ to the equation \eqref{eq}. If there exist two admissible solutions $S$ and $\tilde{S}$ to \eqref{eq}, then there exist some constants $c_1,\cdots,c_{n+1}$, such that $$S(z)-\tilde{S}(z)=\sum_{\alpha=1}^{n+1} c_{\alpha} G(z)(z,E^\alpha).$$
\end{theorem}

We will use  the method of continuity to find an admissible solution of \eqref{eq}. In the next section, we prove the a priori estimates for the equation \eqref{eq}.

\

\section{A priori estimates}

In this section, we shall establish the  a priori estimates for the admissible solution to the equation \eqref{eq}. We will frequently use the symmetric function $\sigma_n(u_{ij})=\det(u_{ij})$.

In view of the Gauss formula \eqref{Gauss},  we have for $z\in\mathcal{W}\subset\mathbb{R}^{n+1}$,
\begin{eqnarray*}
z_{ij}=-z\delta_{ij}-\frac12Q_{ijk}z_k.
\end{eqnarray*}
Hence \begin{eqnarray*}
(G(z)(z,E^\alpha))_{ij}=\frac12Q_{ijk}(G(z)(z,E^\alpha))_k-G(z)(z,E^\alpha)\delta_{ij},
\end{eqnarray*}
which implies \begin{eqnarray}\label{kernel}
L_S(G(z)(z,E^\alpha))=0.
\end{eqnarray}

Hence  for constants $\{a_\alpha\}_{\alpha=1}^{n+1}$, the  function $L(z)=a_\alpha G(z)(z,E^\alpha)$, $z\in \mathcal{W}\subset\mathbb{R}^{n+1}$, satisfies
\begin{eqnarray*}
L_{ij}-\frac12Q_{ijk}L_k+L\delta_{ij}=0.
\end{eqnarray*}
Thus for a solution $S$ of \eqref{eq}, $S+L$ is also a solution.  Such an observation allows us to restrict $S$ to satisfy the following orthogonal condition
\begin{eqnarray}\label{ortho}
\int_\mathcal{W} G(z)(z,E^\alpha) S(z)d\mu=0\quad \forall\alpha=1,2,\cdots, n+1.
\end{eqnarray}
This orthogonal condition means that the origin lies in the interior of the convex body  enclosed by $M$.

Under the restriction \eqref{ortho}, we are able to prove the following a priori estimates.
\begin{theorem}\label{a pri}
For each integral $k\geq 1$ and $\alpha\in (0,1)$, there exist a constant $C$, depending on
$n,k,\alpha, \|K\|_{C^k(\mathcal{W})}, inf_{\mathcal{W}} K,\|F\|_{C^{k+3}(\mathcal{W})}$,
 such that 
\begin{eqnarray*}
\|S\|_{C^{k+1,\alpha}(\mathcal{W})}\leq C,
\end{eqnarray*}
for all admissible solutions of \eqref{eq} satisfying the condition \eqref{ortho}.
\end{theorem} 

We remark that  the norm  of $S$ in this paper are all with respect to the metric $g=G|_\mathcal{W}$ on $\mathcal{W}$.

\subsection{$C^0$ estimate} 
We first establish a uniform positive lower bound and upper bound for $S$. 
Since \eqref{ortho} means that the origin lies in the interior of the convex body  enclosed by $M$, in order to obtain the bounds for $S$, it is sufficient to find the bound for the anisotropic inner and outer radius of $M$ relative to $\mathcal{W}$.

The anisotropic inner radius of $M$ relative to $\mathcal{W}$ is defined as
\begin{eqnarray*}
&& r(M):=\sup\{ t>0: t\mathcal{W}+y\subset\mathcal{K}\hbox{ for some }y\in\mathbb{R}^{n+1}\},
\end{eqnarray*}
and the anisotropic outer radius
of $\mathcal{K}$ relative to $\mathcal{W}$ is defined as
\begin{eqnarray*}
&& R(M):=\inf\{ t>0:  \mathcal{K}\subset t\mathcal{W}+y  \hbox{ for some }y\in\mathbb{R}^{n+1}\}.
\end{eqnarray*}

\begin{lemma}\label{lem1}
Let $M$ be a compact convex $C^2$ hypersurface in $\mathbb{R}^{n+1}$ and $K$ be its anisotropic Gauss-Kronecker curvature function defined on $\mathcal{W}$. Then 
\begin{eqnarray*}
 && \frac12 m_1\leq r(M)\leq R(M)\leq \frac12 m_2,
\end{eqnarray*}
where \begin{eqnarray*}
m_1=2|\mathcal{W}|^{-\frac{1}{n}}\left(\int_\mathcal{W} \frac{1}{K(z)} d\mu\right)^\frac{n+1}{n}\left(\inf_{y\in\mathcal{W}}\int_\mathcal{W}\frac{1}{K(z)} \max\{0,G(z)(z,y)\} d\mu \right)^{-1},
\end{eqnarray*}
\begin{eqnarray*}
&& m_2=2\frac{1}{n+1}|\mathcal{W}|^{-\frac{1}{n}}\left(\int_\mathcal{W}\frac{1}{K(z)}  d\mu\right)^{-1}\cdot\left(\inf_{y\in\mathcal{W}}\int_\mathcal{W}\frac{1}{K(z)} \max\{0,G(z)(z,y)\} d\mu \right)^{\frac{n+1}{n}},
\end{eqnarray*}
and $|\mathcal{W}|$ is the standard $n$-dimensional volume of $\mathcal{W}$.
In particular,
if $S$ is an admissible solution of \eqref{eq} on $\mathcal{W}$ and satisfies \eqref{ortho}, then
\begin{eqnarray*}
0<m_1\leq S\leq m_2.
\end{eqnarray*}
\end{lemma}

\begin{proof}
Since the origin lies in the interior of $M$, we can find $p_0\in M$ with $R_0:=F(p_0)=\max_{p\in M} F(p)$, set $\bar{p}_0=\frac{p_0}{F(p_0)}\in \mathcal{W}$. It is easy to see that $M\subset R_0\mathcal{W}$. Hence $R(M)\leq R_0.$
The support function $S$ at $z\in \mathcal{W}$ satisfies \begin{eqnarray*}
&& S(z)=\sup_{y\in M} G(z)(z,y)\geq \max\{0,G(z)(z,p_0)\}=R_0\max\{0,G(z)(z,\bar{p}_0)\}.
\end{eqnarray*}
Denote $f=\frac{1}{K}$. By multiplying $f$ and integrating over $\mathcal{W}$, we have
\begin{eqnarray}\label{C0eq1}
&& R(M)\leq R_0\leq \left(\int_\mathcal{W} f(z)S(z) d\mu\right)\left(\int_\mathcal{W} f(z)\max\{0,G(z)(z,\bar{p}_0)\} d\mu \right)^{-1}.
\end{eqnarray}
In view of \eqref{vol}, we have a Minkowski formula,
\begin{eqnarray}\label{C0eq2}
&&\int_\mathcal{W} f(z)S(z) d\mu = \int_M S(\nu(X)) d\mu(M)\\&=& \int_M G(\nu)(\nu,X) \Omega(\nu,\partial_1,\cdots,\partial_n)dy^1\cdots dy^n\nonumber\\&=&\int_M  \Omega(X,\partial_1,\cdots,\partial_n)dy^1\cdots dy^n=(n+1)Vol(\bar{M})\nonumber.
\end{eqnarray}
The anisotropic isoperimetric inequality (see Busemann \cite{Bu}) tells that
\begin{eqnarray}\label{C0eq3}
&& Vol(\bar{M})\leq \frac{1}{n+1}|\mathcal{W}|^{-\frac{1}{n}}|M|_F^\frac{n+1}{n}=\frac{1}{n+1}|\mathcal{W}|^{-\frac{1}{n}}\left(\int_\mathcal{W} f(z) d\mu\right)^\frac{n+1}{n}.
\end{eqnarray}

Combining \eqref{C0eq1}, \eqref{C0eq2} and \eqref{C0eq3}, we obtain the upper bound of $R(M)$,
\begin{eqnarray*}
&& R(M)\leq |\mathcal{W}|^{-\frac{1}{n}}\left(\int_\mathcal{W} f(z) d\mu\right)^\frac{n+1}{n}\left(\inf_{y\in\mathcal{W}}\int_\mathcal{W} f(z)\max\{0,G(z)(z,y)\} d\mu \right)^{-1}.
\end{eqnarray*}

Next we want to find the positive lower bound of $r(M)$. Since $\bar{M}$ is enclosed in a rescaled Wulff shape with radius $R(M)$, we have 
\begin{eqnarray}\label{C0eq8}
&&Vol(\bar{M})\leq \frac{1}{n+1}|\mathcal{W}|R(M)^{n+1}.
\end{eqnarray}
It follows from \eqref{C0eq1}, \eqref{C0eq2} and \eqref{C0eq8} that
\begin{eqnarray}\label{C0eq10}
&&R(M)\geq |\mathcal{W}|^{-\frac{1}{n}}\left(\inf_{y\in\mathcal{W}}\int_\mathcal{W} f(z)\max\{0,G(z)(z,y)\} d\mu \right)^{\frac{1}{n}},
\end{eqnarray}
and then
\begin{eqnarray}\label{C0eq11}
&&Vol(\bar{M})\geq \frac{1}{n+1}|\mathcal{W}|^{-\frac{1}{n}}\left(\inf_{y\in\mathcal{W}}\int_\mathcal{W} f(z)\max\{0,G(z)(z,y)\} d\mu \right)^{\frac{n+1}{n}}.
\end{eqnarray}

Recalling an inequality by Ben Andrews \cite{An}, Proposition 5.1, which is a consequence of the Diskant inequalities, 
\begin{eqnarray}\label{C0eq9}
&& r(M)\geq \frac{Vol(\bar{M})}{|M|_F}=\frac{Vol(\bar{M})}{\int_{\mathcal{W}} f(z) d\mu}.
\end{eqnarray}

Combining  \eqref{C0eq11} and \eqref{C0eq9},
we get the positive lower bound of $r(M)$,\begin{eqnarray*}
 r(M)\geq &&\frac{1}{n+1}|\mathcal{W}|^{-\frac{1}{n}}\left(\int_\mathcal{W} f(z) d\mu\right)^{-1}\\&&\cdot\left(\inf_{y\in\mathcal{W}}\int_\mathcal{W} f(z)\max\{0,G(z)(z,y)\} d\mu \right)^{\frac{n+1}{n}}.
\end{eqnarray*}

Now it is easy to derive the upper and positive lower bound of $S$ in terms of $r(M)$ and $R(M)$.
In fact, it follows from Schwarz inequality that 
\begin{eqnarray}
&& S(z)=G(z)(z,X(z))\leq F(z)F(X(z))\leq 2R(M),\quad\forall z\in\mathcal{W}.
\end{eqnarray}
On the other hand, for any $z\in \mathcal{W}$, let $t(z)>0$ be the number such that $t(z)z\in M$. It follows from the difinition of $r(M)$ that $2r(M)\leq \sup_{z\in\mathcal{W}} t(z)$.
Consequently, \begin{eqnarray}
&& S(z)=\sup_{y\in M} G(z)(z,y)\geq G(z)(z,t(z)z)=t(z)\geq 2r(M),\quad\forall z\in\mathcal{W}.
\end{eqnarray}\end{proof}

\subsection{$C^1$ estimate}The next step is a priori $C^1$ estimate for $S$. Such estimate is not necessary in the classical Minkowski problem, since there the $C^2$ estimate is more direct. However, for the equation \eqref{eq}, the gradient term in the determinant causes problem, which cannot be  solved until we have the $C^1$ estimate first. Therefore, in the anisotropic Minkowski problem, the $C^1$ estimate seems necessary. 

\begin{lemma}\label{lem2}
Let $S$ be an admissible solution of \eqref{eq} on $\mathcal{W}$. Let $f=\frac{1}{K}$. Then there exists a constant $C$, depending on $n, m_1,m_2, \max f, \max |\nabla f^\frac1n|, \|F\|_{C^4(\mathcal{W})}$, such that
\begin{eqnarray}
|\nabla S|\leq C.
\end{eqnarray}
\end{lemma}

\begin{proof}
Suppose that $|\nabla S|\geq 1$, otherwise, we are done.
Denote $u_{ij}=S_{ij}- \frac12Q_{ijk}S_{k}+\delta_{ij} S$ and $\mathcal{F}^{ij}=\frac{\partial\sigma_n}{\partial u_{ij}}(u_{ij})$. We know that $\mathcal{F}^{ij}$ is an elliptic operator at an admissible solution.

Let $W=\log |\nabla S|^2+e^{\alpha (m_2-S)}$, with $\alpha>0$ to be chosen later.
Suppose that $W$ attains its maximum at point $z_0\in \mathcal{W}$. Choose an orthonormal basis at $z_0$ such that $u_{ij}$ is diagonal. It is clear that $\mathcal{F}^{ij}$ is also diagonal at $z_0$.

Then at $z_0$, we have
\begin{eqnarray}\label{eq0}
&&0=W_i=\frac{|\nabla S|^2_i}{|\nabla S|^2}-\alpha e^{\alpha(m_2- S)}S_i=\frac{2S_kS_{ki}}{|\nabla S|^2}-\alpha e^{\alpha (m_2- S)}S_i,
\end{eqnarray}

\begin{eqnarray}\label{eq1}
&&0\geq \mathcal{F}^{ij}W_{ij}=\mathcal{F}^{ij}\frac{2S_l S_{lij}+2S_{ki}S_{kj}}{|\nabla S|^2}-\frac{\mathcal{F}^{ij}|\nabla S|^2_i |\nabla S|^2_j}{|\nabla S|^4}\\&&-\alpha  e^{\alpha (m_2- S)} \mathcal{F}^{ij}S_{ij}+\alpha^2 e^{\alpha (m_2- S)}\mathcal{F}^{ij}S_iS_j.\nonumber
\end{eqnarray}

Notice that\begin{eqnarray}\label{eq00}
\mathcal{F}^{ii}u_{ii}=f,\quad \mathcal{F}^{ii}u_{iik}=f_k.
\end{eqnarray}

Using \eqref{eq0} and \eqref{eq00}, we estimate the second term in \eqref{eq1} as follows,
\begin{eqnarray}\label{C1eq1}
&&-\frac{\mathcal{F}^{ij}|\nabla S|^2_i |\nabla S|^2_j}{|\nabla S|^4}= -\mathcal{F}^{ii}\alpha e^{\alpha(m_2- S)}S_i\frac{2S_kS_{ki}}{|\nabla S|^2}\\&=& -2\alpha e^{\alpha(m_2- S)}\frac{\mathcal{F}^{ii}S_iS_k(u_{ki}+ \frac12 Q_{ikl}S_{l}-\delta_{ki} S)}{|\nabla S|^2}\nonumber\\
&\geq & -2\alpha e^{\alpha(m_2- S)}(f+C_1\mathcal{F}^{ii}S_i)+ 2\alpha e^{\alpha(m_2- S)} \frac{S \mathcal{F}^{ii}S_i^2}{|\nabla S|^2}\nonumber\\
&\geq & -2\alpha e^{\alpha(m_2- S)}(f+\frac{\alpha}{4}\mathcal{F}^{ii}S_i^2+ \frac{C_2}{\alpha}\sum_i \mathcal{F}^{ii}),\nonumber
\end{eqnarray}
where we dropped the term $2\alpha e^{\alpha(m_2- S)} \frac{S \mathcal{F}^{ii}S_i^2}{|\nabla S|^2}$ since it is positive, while used the H\"older inequality for the term $\mathcal{F}^{ii}S_i$.

Observe that 
\begin{eqnarray}\label{C1eq2}
S_{lij}=S_{ijl}+R_{mijl}S_m,
\end{eqnarray}
\begin{eqnarray}\label{C1eq3}
\mathcal{F}^{ij}S_{ki}S_{kj}\geq 0.
\end{eqnarray}

By employing \eqref{eq0}, \eqref{eq00}, \eqref{C1eq1}, \eqref{C1eq2} and \eqref{C1eq3} in \eqref{eq1}, we obtain that
\begin{eqnarray}\label{C1eq4}
0 \geq  \mathcal{F}^{ij}W_{ij}&\geq &\frac{2\mathcal{F}^{ij}S_l (u_{ij}+\frac12Q_{ijk}S_{k}-\delta_{ij} S)_l}{|\nabla S|^2}-C_3 e^{\alpha(m_2- S)}\sum_i \mathcal{F}^{ii}\\&&-\alpha e^{\alpha (m_2- S)}\mathcal{F}^{ij}(u_{ij}+\frac12Q_{ijk}S_{k}-\delta_{ij} S)\nonumber\\&& +\frac12\alpha^2e^{\alpha (m_2- S)} \mathcal{F}^{ii}S_i^2 -2\alpha e^{\alpha(m_2- S)}f\nonumber\\
&\geq & -2\frac{|\nabla f|}{|\nabla S|}-3\alpha e^{\alpha(2m_2- S)}f+(\alpha S -C_4)e^{\alpha (m_2- S)}\sum_i \mathcal{F}^{ii}\nonumber\\
&&+\frac12\alpha^2e^{\alpha (m_2- S)} \mathcal{F}^{ii}S_i^2.\nonumber
\end{eqnarray}

Choose $\alpha$ large, such that $\alpha m_1-C_4 \geq 1$. Note that $\sum_i \mathcal{F}^{ii}\geq 1$ (see \cite{Gu}).
It follows from \eqref{C1eq4} that
\begin{eqnarray}\label{eq2}
\sum_i \mathcal{F}^{ii}(1+S_i^2)&\leq &2\frac{|\nabla f|}{|\nabla S|}+3\alpha e^{\alpha(m_2- S)}f\\
&\leq & 2\max |\nabla f|+3\alpha e^{\alpha m_2}\max f.\nonumber
\end{eqnarray}

On the other hand, by using G\aa rding inequality (see G\aa rding \cite{Ga}), we see that 
\begin{eqnarray}\label{eq3}
\sum_i \mathcal{F}^{ii}(1+S_i^2)&\geq & n \left(\prod_i u_{ii}\right)^\frac{n-1}{n}\left(\prod_i (1+S_i^2)\right)^\frac1n\\
&\geq& n f^\frac{n-1}{n}(1+|\nabla S|^2)^\frac1n.\nonumber
\end{eqnarray}
Putting  \eqref{eq3} into \eqref{eq2}, we conclude that
\begin{eqnarray*}
|\nabla S|\leq C(n,m_1,m_2,\max f,\max |\nabla f^\frac{1}{n}|,\|F\|_{C^4(\mathcal{W})}).
\end{eqnarray*}
\end{proof}

\subsection{$C^2$ estimate} The $C^2$ estimate for the classical Minkowski problem is somehow direct and fine due to the structure of its equation. In particular, it involves the  exact formula of Gauss equation. In our problem the gradient term in the determinant also brings troubles. We brings here some idea from Yau's proof \cite{Ya} in Calabi conjecture and Guan-Li's proof \cite{GL} for more general complex Monge-Amp\'ere equations to our equation. It avoids the use of explicit formula for Gauss equation.

\begin{lemma}
Let $S$ be an admissible solution of \eqref{eq} on $\mathcal{W}$ . Let $f=\frac{1}{K}$. Then  there exists a constant $C$, depending on $n, m_1,m_2, \|f\|_{C^2}, \min f,\|F\|_{C^5(\mathcal{W})}$, such that
\begin{eqnarray}
|\nabla^2 S|\leq C.
\end{eqnarray}
\end{lemma}

\begin{proof} Let $\bar{\mathcal{F}}(u_{ij})=\log \sigma_n(u_{ij})$ and $u^{ij}$ be the inverse matrix of $u_{ij}$. Then
\begin{eqnarray}
\bar{\mathcal{F}}^{ij}=f u^{ij}, \quad \bar{\mathcal{F}}^{ij,kl}=-u^{ik}u^{jl}.
\end{eqnarray}

For an admissible solution $S$,  $0<2\sigma_2(u_{ij})=(\sum_{i} u_{ii})^2-\sum_{i,j} |u_{ij}|^2$. In view of Lemma \ref{lem1} and \ref{lem2}, to bound $|\nabla^2 S|$, it is sufficient to bound $\Delta S$ from above. Here $\Delta$ denotes the Laplace operator with respect to $g$. We  may assume that $\Delta S\geq C$ for some $C>0$. 

Let $\Phi=\log (a+\Delta S)+e^{\beta (m_2-S)}$ with $a=\sup_\mathcal{W} |-\frac12 Q_{iik}S_k+nS|$, $\beta>0$ to be chosen later. Suppose that $\Phi$ attains its maximum at point $z_0\in \mathcal{W}$. Choose the orthonormal basis at $z_0$ such that $u_{ij}$ is diagonal. Clearly $\bar{\mathcal{F}}^{ij}$ is also diagonal at $z_0$.
Then we have
\begin{eqnarray}\label{C2eq1}
\Phi_i=\frac{(\Delta S)_i}{a+\Delta S}-\beta e^{\beta (m_2-S)}S_i=0,
\end{eqnarray}
\begin{eqnarray}\label{C2eq2}
&&0\geq \bar{\mathcal{F}}^{ij}\Phi_{ij}=\frac{\bar{\mathcal{F}}^{ii}(\Delta S)_{ii}}{a+\Delta S}-\frac{\bar{\mathcal{F}}^{ii}(\Delta S)_i^2}{(a+\Delta S)^2}-\beta e^{\beta (m_2-S)}\bar{\mathcal{F}}^{ii}S_{ii}+\beta^2 e^{\beta (m_2-S)}\bar{\mathcal{F}}^{ii}S_{i}^2
\end{eqnarray}

We estimate the term $\bar{\mathcal{F}}^{ii}(\Delta S)_{ii}$ and $\bar{\mathcal{F}}^{ii}(\Delta S)_i^2$ by using \eqref{C2eq1} as follows. 
\begin{eqnarray}\label{C2eq3}
&&\bar{\mathcal{F}}^{ii}(\Delta S)_{ii}= \bar{\mathcal{F}}^{ii} (S_{iikk}+Riem*\nabla^2 S+\nabla Riem*\nabla S)\\
&\geq & \bar{\mathcal{F}}^{ii}(u_{ii}+\frac12 Q_{iil}S_l-S)_{kk}- C_5\sum_i \bar{\mathcal{F}}^{ii}(|\nabla^2 S|+|\nabla S|)\nonumber\\
& \geq & \Delta f-\bar{\mathcal{F}}^{ij,rs}u_{ijk}u_{rsk}+\frac12 Q_{iil}(\Delta S)_l- C_6\sum_i \bar{\mathcal{F}}^{ii}(|\nabla^2 S|+|\nabla S|)\nonumber\\
& \geq & u^{ii}u^{jj}u_{ijk}^2+\beta e^{\beta (m_2-S)}\frac12 Q_{iil}S_l(a+\Delta S) -C_7\sum_i\bar{\mathcal{F}}^{ii}\Delta S- C_8(1+\sum_i \bar{\mathcal{F}}^{ii}).\nonumber
\end{eqnarray}
Here and in the following we use  \textit{Riem} to denote the Riemannian tensor of $g$ and the notation $*$ to denote scalar contraction of two tensors by $g$.

\begin{eqnarray}\label{C2eq4}
&&\bar{\mathcal{F}}^{ii}(\Delta S)_i^2=\bar{\mathcal{F}}^{ii}(S_{ikk}+Riem*\nabla S)^2\\
&\leq &\bar{\mathcal{F}}^{ii}S_{ikk}^2+C_9\bar{\mathcal{F}}^{ii}|S_{ikk}|+C_{10}\sum_i \bar{\mathcal{F}}^{ii}\nonumber\\
&\leq & \bar{\mathcal{F}}^{ii}S_{ikk}^2 +C_9 \bar{\mathcal{F}}^{ii}\left((a+\Delta S)|S_i| \beta e^{\beta (2m_2-S)}+|Riem*\nabla S|\right)+C_{10}\sum_i \bar{\mathcal{F}}^{ii}\nonumber\\
&\leq & \bar{\mathcal{F}}^{ii}S_{ikk}^2 +C_{11}(\beta e^{\beta (m_2-S)}|S_i|\Delta S+1) \sum_i \bar{\mathcal{F}}^{ii}.\nonumber
\end{eqnarray}

Since \begin{eqnarray*}
S_{ikk}=u_{ikk}+(\frac12 Q_{ikl}S_l-S\delta_{ik})_k,
\end{eqnarray*}
we have
\begin{eqnarray}\label{C2eq5}
&& S_{ikk}^2-u_{ikk}^2= 2u_{ikk}(\frac12 Q_{ikl}S_l-S\delta_{ik})_k+(\frac12 Q_{ikl}S_l-S\delta_{ik})_k^2\\
&=& 2S_{ikk}(\frac12 Q_{ikl}S_l-S\delta_{ik})_k-(\frac12 Q_{ikl}S_l-S\delta_{ik})_k^2\nonumber\\
&=& 2\left((a+\Delta S) \beta e^{\beta (m_2-S)}S_i+Riem*\nabla S\right)(\frac12 Q_{ikl}S_l-S\delta_{ik})_k-(\frac12 Q_{ikl}S_l-S\delta_{ik})_k^2\nonumber\\
&\leq & C_{12}(1+\beta e^{\beta (m_2-S)}|S_i|)\left((\Delta S)^2+\Delta S+1\right).\nonumber
\end{eqnarray}

Now using \eqref{C2eq4} and \eqref{C2eq5}, we obtain
\begin{eqnarray}\label{C2eq6}
&&\frac{\bar{\mathcal{F}}^{ii}(\Delta S)_i^2}{(a+\Delta S)^2}\leq \frac{\bar{\mathcal{F}}^{ii}S_{ikk}^2}{(a+\Delta S)^2}+C_{13}(\beta e^{\beta (m_2-S)}|S_i|+1) \sum_i \bar{\mathcal{F}}^{ii}\\
&\leq & \frac{\bar{\mathcal{F}}^{ii}u_{ikk}^2}{(a+\Delta S)^2}+ C_{14}(\beta e^{\beta (m_2-S)}|S_i|+1)\sum_i \bar{\mathcal{F}}^{ii}\nonumber\\
&\leq & \frac{\bar{\mathcal{F}}^{ii}u_{ikk}^2}{(a+\Delta S)^2}+\frac12 C_{14} \beta^{\frac32}e^{\beta (m_2-S)}\bar{\mathcal{F}}^{ii}S_i^2   +\frac12 C_{14}\beta^\frac12 e^{\beta (m_2-S)}\sum_i \bar{\mathcal{F}}^{ii}+C_{14}\sum_i \bar{\mathcal{F}}^{ii}.\nonumber
\end{eqnarray}
In the last inequality, we used the H\"older inequality.

\

We estimate the term $\bar{\mathcal{F}}^{ii}u_{ikk}^2$ by Cauchy-Schwarz inequality,
\begin{eqnarray}\label{C2eq7}
\bar{\mathcal{F}}^{ii}u_{ikk}^2&\leq &\sum_{i,j,k} u^{ii}(u_{jj}u^{jj}u_{ijk}^2)\\
&\leq & \sum_j u_{jj} \sum_{i,j,k}u^{ii}u^{jj}u_{ijk}^2\nonumber\\
&\leq & (a +\Delta S)\sum_{i,j,k}u^{ii}u^{jj}u_{ijk}^2.\nonumber
\end{eqnarray}

It follows from \eqref{C2eq3}, \eqref{C2eq6} and \eqref{C2eq7} that
\begin{eqnarray}\label{C2eq8}
\frac{\bar{\mathcal{F}}^{ii}(\Delta S)_{ii}}{a+\Delta S}-\frac{\bar{\mathcal{F}}^{ii}(\Delta S)_i^2}{(a+\Delta S)^2}&\geq &\beta e^{\beta (m_2-S)}\frac12 \bar{\mathcal{F}}^{ii}Q_{iil}S_l-C_{15} \beta^{\frac32}e^{\beta (m_2-S)}\bar{\mathcal{F}}^{ii}S_i^2 \\&&   -C_{16}\beta^\frac12 e^{\beta (m_2-S)}\sum_i \bar{\mathcal{F}}^{ii}-C_{17}\sum_i \bar{\mathcal{F}}^{ii}.\nonumber
\end{eqnarray}

We also have 
\begin{eqnarray}\label{C2eq9}
&&-\beta e^{\beta (m_2-S)}\bar{\mathcal{F}}^{ij}S_{ij}=-\beta e^{\beta (m_2-S)}\bar{\mathcal{F}}^{ii}(u_{ii}+\frac12 Q_{iil}S_l-S)\\&\geq & -\beta e^{\beta (m_2-S)}f-\beta e^{\beta (m_2-S)}\frac12 \bar{\mathcal{F}}^{ii} Q_{iil}S_l+m_1\beta e^{\beta (m_2-S)}\sum_i \bar{\mathcal{F}}^{ii},\nonumber
\end{eqnarray}
where $m_1$ is the uniform positive bound of $S$ in Lemma \ref{lem1}.

Therefore, by combining \eqref{C2eq2}, \eqref{C2eq8} and \eqref{C2eq9}, we get
\begin{eqnarray}\label{C2eq10}
0&\geq &e^{\beta (m_2-S)}(m_1\beta - C_{16}\beta^\frac12 - C_{18} )\sum_i \bar{\mathcal{F}}^{ii}\\&&+e^{\beta (m_2-S)} (\beta^2- C_{15} \beta^{\frac32})\bar{\mathcal{F}}^{ii}S_i^2 -C_{19}.\nonumber
\end{eqnarray}

Choose $\beta$ large enough in \eqref{C2eq10}, we obtain 
\begin{eqnarray}\label{C2eq11}
\sum_i \bar{\mathcal{F}}^{ii}\leq C,
\end{eqnarray}
where $C$ depends on $m_1, m_2, \|f\|_{C^2}, \|S\|_{C^1},\|F\|_{C^5(\mathcal{W})}$.

Recall the following elementary inequality (see Yau \cite{Ya}), 
\begin{eqnarray*}
\frac{\sum_i \lambda_i}{\prod_i \lambda_i}\leq \left(\sum_i \lambda_i^{-1}\right)^{n-1}\hbox{ for }\lambda_i>0,\forall i=1,\cdots, n.
\end{eqnarray*}

Since $\bar{\mathcal{F}}^{ii}=u^{ii}=u_{ii}^{-1}$, $\det(u_{ii})=f$ we use the previous inequality by $\lambda_i=u_{ii}$ and \eqref{C2eq11} to get
\begin{eqnarray*}
\sum_i u_{ii}\leq C_{20}\prod_i u_{ii}\leq C_{21}
\end{eqnarray*}
which implies \begin{eqnarray*}
\Delta S\leq C
\end{eqnarray*}
where $C$ depends on $m, M, \|f\|_{C^2},\|S\|_{C^1},\|F\|_{C^5(\mathcal{W})}$.
Note that our computation is only valid at the maximum point $z_0$ of $\Phi$. Nevertheless, we have 
\begin{eqnarray*}
\Phi\leq \Phi(z_0)=\log (a+\Delta S(z_0))+\beta e^{\beta(m_2-S(z_0))} \leq C,
\end{eqnarray*}
which yields the $C^2$ estimate for $S$ at every point.
\end{proof}

\noindent{\it Proof of Theorem \ref{a pri}:} Once we have $C^2$ estimate, Theorem \ref{a pri} follows from the Evans-Krylov theorem and the standard elliptic theory.\qed

\

\section{Openness  and proof of Theorem \ref{thm1}}

By virtue of Theorem \ref{a pri}, we can assume that $K\in C^\infty(\mathcal{W})$. We will use the method of continuity to prove Theorem \ref{thm1}.   To be precise, let 
\begin{eqnarray*}
\frac{1}{K_t}=\frac{t}{K}+(1-t) \quad\forall 0\leq t\leq 1.
\end{eqnarray*}
It is easy to see that $K_t>0$ and satisfies \eqref{necc}.
Define 
$$\mathcal{S}=\{t\in[0,1]|\det(S_{ij}-\frac12Q_{ijk}S_k+S\delta_{ij})=\frac{1}{K_t} \hbox{ has an admissible solution  on }\mathcal{W} \}.$$
Clearly, $0\in \mathcal{S}$ since $\mathcal{W}$ has anisotropic Gauss curvature $1$ (see Example \ref{exam}). 
We will apply the implicit function theorem to \eqref{eq} to prove the openness of the set $\mathcal{S}$.

\begin{proposition}\label{open}
The set $\mathcal{S}$ is open in $[0,1]$.
\end{proposition}

To prove the openness of $\mathcal{S}$, we shall study the  linearized operator of $S\mapsto \det(S_{ij}-\frac12Q_{ijk}S_k+S\delta_{ij}).$

Denote $u_{ij}=S_{ij}-\frac12Q_{ijk}S_k+S\delta_{ij}$. Let $L_S$ be the linearized operator of $S\mapsto \det(u_{ij})$, namely, \begin{eqnarray*}
L_S(v)=\frac{\partial\sigma_n}{\partial u_{ij}}(u_{ij}) (v_{ij}-\frac12Q_{ijk}v_k+v\delta_{ij})
\end{eqnarray*}
for any $v\in C^\infty(\mathcal{W})$.
The following proposition shows that $L_S$ is a self-adjoint operator.
\begin{lemma}\label{adjoint}
For any $S,v,w\in C^2(\mathcal{W})$, we have\begin{eqnarray*}
\int_\mathcal{W}w L_S(v) d\mu=\int_\mathcal{W}v  L_S(w)d\mu.
\end{eqnarray*}\end{lemma}
\begin{proof}
From the Gauss equation \eqref{Gausseq1}, we have
\begin{eqnarray}\label{adeq}
&& S_{ijk}-S_{ikj}=R_{pijk}S_p= S_j\delta_{ik}-S_k\delta_{ij}+\frac14 (Q_{ijm} Q_{mkp}-Q_{ikm} Q_{mjp})S_p.
\end{eqnarray}
By using \eqref{4Q}, namely $Q_{ijk,l}=Q_{ilk,j}$ and \eqref{adeq}, we see that 
\begin{eqnarray}\label{adeq0}
&& u_{ijk}-u_{ikj}= \frac14 (Q_{ijm} Q_{mkp}-Q_{ikm} Q_{mjp})S_p-\frac12 Q_{ijm}S_{mk}+\frac12 Q_{ikm}S_{mj}\\
&=&-\frac12 Q_{ijm}(u_{mk}-S_{mk}-S\delta_{mk})+\frac12 Q_{ikm}(u_{mj}-S_{mj}-S\delta_{mj})\nonumber\\&&
-\frac12 Q_{ijm}S_{mk}+\frac12 Q_{ikm}S_{mj}\nonumber\\
&=&\frac12 Q_{ikm}u_{mj}-\frac12 Q_{ijm}u_{mk}.\nonumber
\end{eqnarray}

By definition, \begin{eqnarray*}
\sigma_n(u_{ij})=\frac{1}{n!}\sum_{\substack{i_1,\cdots i_{n}=1,\\ j_1,\cdots,j_{n}=1}}^n\delta_{j_1\cdots j_n}^{i_1\cdots i_n}u_{i_1j_1}\cdots u_{i_nj_n}.
\end{eqnarray*}
\begin{eqnarray*}
\frac{\partial\sigma_n}{\partial u_{ij}}(u_{ij})=\frac{1}{(n-1)!}\sum_{\substack{i_1,\cdots i_{n-1}=1,\\ j_1,\cdots,j_{n-1}=1}}^n\delta_{j_1\cdots j_{n-1} j}^{i_1\cdots i_{n-1} i}u_{i_1j_1}\cdots u_{i_{n-1}j_{n-1}}.
\end{eqnarray*}
Here $\delta_{j_1\cdots j_n}^{i_1\cdots i_n}$ denotes the Kronecker symbols, i.e., it equals to $1$ ($-1$ reps.) if $(i_1\cdots i_n)$ is an even (odd reps.) permutation of  $(j_1\cdots j_n)$ and it equals to 0 in other cases.

Thus using the antisymmetry of the Kronecker symbols and \eqref{adeq0},  we obtain
\begin{eqnarray}\label{ad eq0}
&&\sum_{j=1}^n \left(\frac{\partial\sigma_n}{\partial u_{ij}}(u_{ij})\right)_j\\&=&\frac{1}{2(n-1)!}\sum_{\substack{i_1,\cdots i_{n-1}=1,\\ j,j_1,\cdots,j_{n-1}=1}}^n (n-1)\delta_{j_1\cdots  j_{n-1} j}^{i_1\cdots i_{n-1} i}(u_{i_1j_1j}-u_{i_1 jj_1})u_{i_2j_2}\cdots u_{i_{n-1}j_{n-1}}\nonumber\\
&=&\frac{1}{2(n-2)!}\sum_{\substack{i_1,\cdots i_{n-1}=1,\\ j,j_1,\cdots,j_{n-1}=1}}^n \delta_{j_1\cdots  j_{n-1} j}^{i_1\cdots i_{n-1} i}\left(\frac12 Q_{i_1 j m}u_{m j_1}-\frac12 Q_{i_1 j_1 m} u_{mj}\right)u_{i_2j_2}\cdots u_{i_{n-1}j_{n-1}}.\nonumber
\end{eqnarray}
We assume $u_{ij}$ is diagonal at some point $p_0$, i.e., $u_{ij}=\lambda_i \delta_{ij}$. We will use the notation $P(1,\cdots,n)$ to denote the permutation group of $\{1,\cdots,n\}$ and similarly, $P(1,\cdots,\hat{i},\cdots,n)$ the  permutation group of $\{1,\cdots,n\}$ without index $i$. 

Then at the point $p_0$, \eqref{ad eq0} reduces to
\begin{eqnarray}\label{ad eq4}
&&\sum_{j=1}^n \left(\frac{\partial\sigma_n}{\partial u_{ij}}(u_{ij})\right)_j\\&=&\frac{1}{2(n-2)!}\sum_{j,i_1,j_1=1}^{n}\delta_{j_1 j}^{i_1 i} \left(\frac12 Q_{i_1 j j_1}\lambda_{j_1}-\frac12 Q_{i_1 j_1 j} \lambda_{j}\right)\sum_{\substack{(i_2,\cdots, i_{n-1})\in\\ P(1,\cdots,\hat{i},\hat{i_1},\cdots,n)}}\lambda_{i_2}\cdots  \lambda_{i_{n-1}}\nonumber\\&=&\frac{1}{2(n-2)!} \sum_{j=i,j_1=i_1\neq i}^n\left(\frac12 Q_{i_1 i i_1}\lambda_{i_1}-\frac12 Q_{i_1 i_1 i} \lambda_{i}\right)\sum_{\substack{(i_2,\cdots, i_{n-1})\in\\ P(1,\cdots,\hat{i},\hat{i_1},\cdots,n)}} \lambda_{i_2}\cdots \lambda_{i_{n-1}}\nonumber\\&&
+\frac{1}{2(n-2)!} \sum_{j_1=i,j=i_1\neq i} (-1) \left(\frac12 Q_{i_1 i_1 i}\lambda_{i}-\frac12 Q_{i_1 i i_1} \lambda_{i_1}\right)\sum_{\substack{(i_2,\cdots, i_{n-1})\in\\ P(1,\cdots,\hat{i},\hat{i_1},\cdots,n)}}\lambda_{i_2}\cdots \lambda_{i_{n-1}}\nonumber\\
&=&\frac{1}{(n-2)!}\sum_{\substack{(i_2,\cdots, i_{n-1})\in\\ P(1,\cdots,\hat{i},\hat{i_1},\cdots,n)}}  \left(\frac12 Q_{i_1 i i_1} \lambda_{i_1}\cdots \lambda_{i_{n-1}}-\frac12  Q_{i_1 i_1 i} \lambda_{i}\lambda_{i_2}\cdots \lambda_{i_{n-1}}\right).\nonumber\end{eqnarray}
On the other hand,  at $p_0$, we have \begin{eqnarray}\label{ad eq5}
\frac{\partial\sigma_n}{\partial u_{ij}}(u_{ij})&=&\frac{1}{(n-1)!}\sum_{\substack{(i_1,\cdots, i_{n-1})\in \\P(1,\cdots,\hat{i},\cdots,n)}} \lambda_{i_1}\cdots \lambda_{i_{n-1}} \delta_{ij}.
\end{eqnarray}
Recall from Lemma \ref{volume} that $d\mu=\varphi dV_g$ and $ \varphi_i=\frac12 \sum_{k=1}^n Q_{ikk}\varphi$. By combining \eqref{ad eq4} and \eqref{ad eq5}, we obtain at $p_0$,
\begin{eqnarray}\label{ad eq1}
&&\left(\frac{\partial\sigma_n}{\partial u_{ij}}(u_{ij})\right)_j v_i\varphi+\frac{\partial\sigma_n}{\partial u_{ij}}(u_{ij})\left(v_i\varphi_j+\frac12 Q_{ijk}v_k\varphi\right)\\
&=&\frac{1}{(n-2)!}\sum_{i=1}^n \sum_{\substack{(i_1,\cdots, i_{n-1})\in \\P(1,\cdots,\hat{i},\cdots,n)}}  \left(\frac12 Q_{i_1 i i_1} \lambda_{i_1}\cdots \lambda_{i_{n-1}}-\frac12  Q_{i_1 i_1 i} \lambda_{i}\lambda_{i_2}\cdots \lambda_{i_{n-1}}\right) v_i\varphi\nonumber\\
&&+\frac{1}{(n-1)!}\sum_{i=1}^n \sum_{\substack{(i_1,\cdots, i_{n-1})\in \\P(1,\cdots,\hat{i},\cdots,n)}}\lambda_{i_1}\cdots \lambda_{i_{n-1}}\left(v_i (-\frac12 \sum_{k=1}^n Q_{ikk}\varphi)+\sum_{k=1}^n\frac12 Q_{iik}v_k\varphi\right).\nonumber
\end{eqnarray}
It is easy to see that the second term in the right hand side of \eqref{ad eq1} can be viewed as
\begin{eqnarray}\label{ad eq7}
&&\frac{1}{(n-1)!}\sum_{i=1}^n \sum_{k\neq i}\sum_{\substack{(i_1,\cdots, i_{n-1})\in \\P(1,\cdots,\hat{i},\cdots,n)}}\lambda_{i_1}\cdots \lambda_{i_{n-1}}\left(v_i (-\frac12 \sum_{k=1}^n Q_{ikk}\varphi)+\sum_{k=1}^n\frac12 Q_{iik}v_k\varphi\right).
\end{eqnarray}
A simple computation shows that
\begin{eqnarray}\label{ad eq2}
&&(n-1)\sum_{i=1}^n \sum_{\substack{(i_1,\cdots, i_{n-1})\in\\ P(1,\cdots,\hat{i},\cdots,n)}}  \frac12 Q_{i_1 i i_1} \lambda_{i_1}\cdots \lambda_{i_{n-1}}v_i\varphi\\&&= \sum_{i=1}^n \sum_{k\neq i}\sum_{\substack{(i_1,\cdots, i_{n-1})\in\\ P(1,\cdots,\hat{i},\cdots,n)}} \frac12 Q_{ikk}\lambda_{i_1}\cdots \lambda_{i_{n-1}}v_i  \varphi.\nonumber
\end{eqnarray}
\begin{eqnarray}\label{ad eq3}
&&(n-1)\sum_{i=1}^n \sum_{\substack{(i_1,\cdots, i_{n-1})\in\\ P(1,\cdots,\hat{i},\cdots,n)}}  \frac12  Q_{i_1 i_1 i} \lambda_{i}\lambda_{i_2}\cdots \lambda_{i_{n-1}} v_i\varphi\\&&= \sum_{i=1}^n \sum_{k\neq i}\sum_{\substack{(i_1,\cdots, i_{n-1})\in\\ P(1,\cdots,\hat{i},\cdots,n)}}\frac12 Q_{iik}\lambda_{i_1}\cdots \lambda_{i_{n-1}}v_k\varphi.\nonumber
\end{eqnarray}
Substituting \eqref{ad eq7}, \eqref{ad eq2} and \eqref{ad eq3} into \eqref{ad eq1}, we see that
\begin{eqnarray}\label{ad eq6}
\left(\frac{\partial\sigma_n}{\partial u_{ij}}(u_{ij})\right)_j v_i\varphi+\frac{\partial\sigma_n}{\partial u_{ij}}(u_{ij})\left(v_i\varphi_j+\frac12 Q_{ijk}v_k\varphi\right)=0.
\end{eqnarray}
Since at every point we can choose a local normal coordinate such that $u_{ij}$ is diagonal, \eqref{ad eq6} holds for any points in $\mathcal{W}$.

With the help of \eqref{ad eq6}, we are easy to achieve the lemma. Indeed, integrating by parts, we have
\begin{eqnarray*}
&&\int_\mathcal{W} w L_S(v)d\mu\\&=&\int_\mathcal{W} w \frac{\partial\sigma_n}{\partial u_{ij}}(u_{ij})(v_{ij}-\frac12 Q_{ijk}v_k+v\delta_{ij})\varphi dV_g\nonumber\\&=&-\int_\mathcal{W}  \frac{\partial\sigma_n}{\partial u_{ij}}(u_{ij})w_iv_j \varphi dV_g+\int_\mathcal{W}  \frac{\partial\sigma_n}{\partial u_{ij}}(u_{ij})\delta_{ij}wv\varphi dV_g\nonumber\\&&-\int_\mathcal{W} w\left\{ \left(\frac{\partial\sigma_n}{\partial u_{ij}}(u_{ij})\right)_j v_i\varphi+\frac{\partial\sigma_n}{\partial u_{ij}}(u_{ij})\left(v_i\varphi_j+\frac12 Q_{ijk}v_k\varphi\right) \right\}dV_g\nonumber\\
&=&-\int_\mathcal{W}  \frac{\partial\sigma_n}{\partial u_{ij}}(u_{ij})w_iv_j d\mu+\int_\mathcal{W}  \frac{\partial\sigma_n}{\partial u_{ij}}(u_{ij})\delta_{ij}wvd\mu,\nonumber
\end{eqnarray*}
which is symmetric in $w$ and $v$. We finish the proof of Lemma \ref{adjoint}.
\end{proof}

The following is an immediate consequence of Lemma \ref{adjoint}. It also generalizes the necessary condition \eqref{necc} to any functions $S\in C^2(\mathcal{W})$, not just the anisotropic support function for some convex hypersurface.
\begin{corollary}Let $S\in C^2(\mathcal{W})$. Let $E^\alpha$ denote the standard $\alpha$-th coordinate vector of $\mathbb{R}^{n+1}$. For the position vector $z\in \mathcal{W}\subset \mathbb{R}^{n+1}$, we have
\begin{eqnarray*}
\int_\mathcal{W}  G(z)(z,E^\alpha) \det(S_{ij}-\frac12Q_{ijk}S_k+S\delta_{ij})d\mu=0,  \forall \alpha=1,\cdots,n+1.
\end{eqnarray*}\end{corollary}
\begin{proof}
Recall the equation \eqref{kernel},
\begin{eqnarray*}
L_S(G(z)(z,E^\alpha))=0.
\end{eqnarray*}
Then it follows from Proposition \ref{adjoint} that
\begin{eqnarray*}
&&\int_\mathcal{W}  G(z)(z,E^\alpha) \det(S_{ij}-\frac12Q_{ijk}S_k+S\delta_{ij})d\mu\\
&=&\frac1n \int_\mathcal{W}  G(z)(z,E^\alpha) L_S(S)d\mu=\frac1n \int_\mathcal{W}  L_S(G(z)(z,E^\alpha))  Sd\mu=0.
\end{eqnarray*}
\end{proof}

We observe from \eqref{kernel} that the function space $\hbox{span}\{G(z)(z,E^1),\cdots,G(z)(z,E^{n+1})\}$ lies in the kernel of $L_S$. Next we show that the kernel of $L_S$ contains only  the functions in $\hbox{span}\{G(z)(z,E^1),\cdots,G(z)(z,E^{n+1})\}$.
\begin{lemma}\label{ker}
Let $v\in C^2(\mathcal{W})$ be a function such that $L_S(v)=0$. Then, $v= \sum_{\alpha=1}^{n+1} a_\alpha G(z)(z,E^\alpha)$ for some constants $a_1,\cdots,a_{n+1}$.
\end{lemma}
\begin{proof}
We follow the idea of Cheng-Yau's proof. Let $e_{n+1}=z$ be the position vector of $\mathcal{W}$ and $\{e_1,\cdots,e_n\}$ is a local orthonormal frame field with respect to $g$ on $\mathcal{W}$ such that  $\{e_1,\cdots,e_{n+1}\}$ is a positive oriented  orthonormal frame field with respect to $G$ in $\mathbb{R}^{n+1}$. Let $\{\omega^1,\cdots,\omega^{n+1}\}$ be the dual $1$-form of $\{e_1,\cdots,e_{n+1}\}$, i.e., $\omega^\alpha(e_\beta)=\delta_{\alpha\beta}$. 
Clearly we have
\begin{eqnarray*}
\omega^{n+1}|_\mathcal{W}=0 , \quad e_i(e_{n+1})=e_i,
\end{eqnarray*}
\begin{eqnarray*}
e_i(e_j)=\nabla_{e_i}e_j-\frac12 Q_{ijk}e_k-\delta_{ij}e_{n+1}.
\end{eqnarray*}

Consider the vector valued function $Z=\sum_{i=1}^n v_ie_i+ve_{n+1}$. Then $v=G(z)(z,Z)$ and on $\mathcal{W}$ ( we compute at normal coordinates, namely, $\nabla_{e_i} e_j=0$),
\begin{eqnarray}\label{ker eq1}
dZ&=&e_j(Z)\omega^j=\left[e_je_i( v) e_i+v_i e_j(e_i)+e_j( v)e_{n+1}+ve_j(e_{n+1})\right]\omega^j\\
&=&\left[v_{ij}e_i+v_i(-\frac12 Q_{ijk}e_k-\delta_{ij}e_{n+1})+v_j e_{n+1}+ve_j\right]\omega^j\nonumber\\
&=&\left[(v_{ij}-\frac12 Q_{ijk}v_k+\delta_{ij} v)e_i\right]\omega^j.\nonumber
\end{eqnarray}
Let $X=\sum_{i=1}^n S_ie_i+Se_{n+1}$, then the same computation as \eqref{ker eq1} gives
\begin{eqnarray*}
dX=\left[(S_{ij}-\frac12 Q_{ijk}S_k+\delta_{ij} S)e_i\right]\omega^j.
\end{eqnarray*}

Consider the $(n-1)$-form $\bar{\Omega}=X\wedge Z\wedge dZ\wedge dX\wedge\cdots\wedge dX$, where $dX$ appears $(n-2)$ times.
Since $L_S(v)=0$, we see that
\begin{eqnarray*}
&& dX\wedge Z\wedge dZ\wedge dX\wedge\cdots\wedge dX\\
&=& \left[\frac{\partial \sigma_n}{u_{ij}}(u_{ij})(v_{ij}-\frac12 Q_{ijk}v_k+\delta_{ij} v)\right](Z\wedge e_1\wedge\cdots\wedge e_n)\otimes (\omega^1\wedge \cdots \wedge\omega^n)=0. \nonumber
\end{eqnarray*}
Hence we have
\begin{eqnarray*}
0=\int_\mathcal{W} d\bar{\Omega}=\int_{\mathcal{W}}  X\wedge dZ\wedge dZ\wedge dX\wedge\cdots\wedge dX.
\end{eqnarray*}
The same argument as in \cite{CY}, Page 507, leads to the conclusion that 
\begin{eqnarray*}
v_{ij}-\frac12 Q_{ijk}v_k+\delta_{ij} v=0 \quad\forall i,j=1,\cdots,n.
\end{eqnarray*}
Thus $Z$ is constant due to \eqref{ker eq1} and can be written as $Z= a_\alpha E^\alpha$ for some constants $a_\alpha$. Consequently,\begin{eqnarray*}
v=G(z)(z,Z)=a_\alpha G(z)(z,E^\alpha).
\end{eqnarray*}

\end{proof}

Now we are ready to prove Proposition \ref{open} and then Theorem \ref{thm1} and \ref{thm}.

\

\noindent{\it Proof of Proposition \ref{open}:} Without loss of generality, we assume that  $S$ satisfies \eqref{ortho}. By virtue of Proposition \ref{a pri}
we may further assume that $K\in C^\infty(\mathcal{W})$.  Let $ H^m(\mathcal{W})$ be the Sobolev space of $\mathcal{W}$ with the Riemannian metric $g$. Choose $m$ sufficient large such that $H^m(\mathcal{W})\subset C^4(\mathcal{W})$.
Consider $L_S$ as a bounded linear map from $H^{m+2}(\mathcal{W})$ to $H^{m}(\mathcal{W})$. It follows from Lemma \ref{ker} that $\hbox{Ker}(L_S)=\hbox{span}\{G(z)(z,E^1),\cdots,G(z)(z,E^{n+1})\}$. On the other hand, $L_S$ is self-adjoint due to Lemma \ref{adjoint}. Hence by the standard Hilbert space theory, we have $$\hbox{Image}(L_S)=\hbox{Ker}(L_S^*)^\perp=\hbox{span}\{G(z)(z,E^1),\cdots,G(z)(z,E^{n+1})\}^\perp.$$
Consequently, for any $f\in H^m(\mathcal{W})$ with $\int G(z)(z,E^\alpha)f(z)d\mu=0,\quad\forall \alpha=1,\cdots,n+1$, we have $f\in \hbox{Image}(L_S)$, which means $L_S:H^{m+2}(\mathcal{W})\to H^{m}(\mathcal{W})$ is surjective. The standard implicit function theorem yields that the operator $S\mapsto \det(S_{ij}-\frac12 Q_{ijk}S_k+S\delta_{ij})$ is locally invertible near $S$, which implies the set $\mathcal{S}$ is open.\qed

\

\noindent{\it Proof of Theorem \ref{thm1}:} We see from Theorem \ref{a pri} that $\mathcal{S}$ is closed. Since $\mathcal{S}$ is also open and non-empty, we conclude that $\mathcal{S}=[0,1].$ In particular, \eqref{eq} has an admissible solution on $\mathcal{W}$.

We now turn to the uniqueness part. Assume $S$ and $\tilde{S}$ are two solutions to \eqref{eq}. Denote  by $ U=(u_{ij})=(S_{ij}-\frac12 Q_{ijk}S_k+S\delta_{ij})$ and $\tilde{U}=(\tilde{u}_{ij})=(\tilde{S}_{ij}-\frac12 Q_{ijk}\tilde{S}_k+\tilde{S}\delta_{ij})$.

For any $n\times n$ symmetric matrices $W_1,\cdots,W_n$. Let $\sigma_n(W_1,\cdots,W_n)$ denote the complete polarization of $\sigma_n$, i.e., \begin{eqnarray*}
\sigma_n(W_1,\cdots,W_n)=\frac{1}{n!}\frac{\partial^n}{\partial \lambda_1\cdots\partial \lambda_n}\big|_{\lambda_1=\cdots=\lambda_n=0}\left(\sigma_n(\lambda_1 W_1+\cdots+\lambda_n W_n)\right).
\end{eqnarray*}
Clearly, $$\frac{\partial \sigma_n}{\partial u_{ij}}(u_{ij})\tilde{u}_{ij}=n\sigma_n(\underbrace{U,\cdots,U,}_{(n-1)\hbox{ times }}\tilde{U}).$$

It follows from Lemma \ref{adjoint} that
\begin{eqnarray}\label{unieq1}
&&\int_\mathcal{W} S\sigma_n(\tilde{U})d\mu=\int_\mathcal{W} S \cdot\frac1n L_{\tilde{S}}(\tilde{S})d\mu\\&=&\int_\mathcal{W} \frac1n L_{\tilde{S}}(S)\tilde{S}d\mu=\int_\mathcal{W} \tilde{S}\sigma_n(\tilde{U},\cdots,\tilde{U}, U)d\mu.\nonumber
\end{eqnarray}
In the same way, we have
\begin{eqnarray}\label{unieq2}
\int_\mathcal{W} \tilde{S}\sigma_n(U)d\mu=\int_\mathcal{W} S\sigma_n(U,\cdots,U,\tilde{U})d\mu.
\end{eqnarray}
Combining \eqref{unieq1} and \eqref{unieq2}, we obtain
\begin{eqnarray}\label{unieq3}
&&2\int_\mathcal{W} S\left(\sigma_n(\tilde{U})-\sigma_n(U,\cdots,U,\tilde{U})\right)d\mu\\&=&\int_\mathcal{W} \left[\tilde{S}\left(\sigma_n(U,\tilde{U},\cdots,\tilde{U})-\sigma_n(U)\right)-S\left(\sigma_n(U,\cdots,U,\tilde{U})-\sigma_n(\tilde{U})\right)\right]d\mu.\nonumber
\end{eqnarray}
Recall the G\aa rding inequality for the polarizations of $\sigma_n$ (see G\aa rding \cite{Ga}),
\begin{eqnarray}\label{unieq4}
\sigma_n(U,\cdots,U,\tilde{U})\geq \sigma_n^{\frac1n}(\tilde{U})\sigma_n^{\frac{n-1}{n}}(U),\end{eqnarray}
with the equality holds if and only if $U$ and $\tilde{U}$ are proportional.
In view of the assumption that $\sigma_n(U)=\sigma_n(\tilde{U})$, we see from \eqref{unieq4} that the left hand side of \eqref{unieq3} is non-positive, whence the right hand side is also non-positive. However, the right hand side of \eqref{unieq3} is anti-symmetric with respect to $S$ and $\tilde{S}$, we conclude that it vanishes. This implies the equality holds in \eqref{unieq4}, and in turn, $U$ and $\tilde{U}$ are proportional. Since $\sigma_n(U)=\sigma_n(\tilde{U})$, we obtain that $U=\tilde{U}$ for every point in $\mathcal{W}$.
In particular,
\begin{eqnarray*}
L_S (S-\tilde{S})=n(\sigma_n(U)-\sigma_n(U,\cdots,U,\tilde{U}))=0.
\end{eqnarray*}
By Lemma \ref{ker}, we conclude that $S-\tilde{S}= c_\alpha G(z)(z, E^\alpha).$ The proof is completed.\qed

\

\noindent{\it Proof of Theorem \ref{thm}:} The existence part follows directly from Theorem \ref{thm1} and Propostion \ref{P1} and \ref{P2}. Notice that for two hypersurface $M$ and $\tilde{M}$ with the same anisotropic Gauss-Kronecker curvature,
$$G(z)(z,M(z)-\tilde{M}(z))=S(M(z))-\tilde{S}(M(z))=c_\alpha G(z)(z, E^\alpha).$$
Therefore, $M(z)-\tilde{M}(z)=c_\alpha E^\alpha$, which is a constant vector in $\mathbb{R}^{n+1},$ namely, $M$ and $\tilde{M}$ coincide up to a translation.\qed

 \
 
  \noindent\textbf{Acknowledgement.} I would like to express my gratitude to my advisor, Prof. Guofang Wang for introducing me the topics of anisotropy and  for helpful discussions and suggestions on this paper. 
  
\

\end{document}